\documentclass[12pt,reqno]{amsart}
\usepackage{amssymb,amsmath,amsthm,fullpage,enumerate}
\usepackage{xcolor}
\usepackage{biblatex} 
\usepackage{comment}
\usepackage{cases}

\usepackage[OT2,OT1]{fontenc}

\DeclareFontFamily{OT2}{cmr}{\hyphenchar\font45 }
\DeclareFontShape{OT2}{cmr}{m}{n}{<->wncyr10}{}
\DeclareFontShape{OT2}{cmr}{m}{it}{<->wncyi10}{}
\DeclareFontShape{OT2}{cmr}{m}{sc}{<->wncysc10}{}
\DeclareFontShape{OT2}{cmr}{b}{n}{<->wncyb10}{}
\DeclareFontShape{OT2}{cmr}{bx}{n}{<->ssub*wncyr/b/n}{}

\DeclareFontFamily{OT2}{cmss}{\hyphenchar\font45 }
\DeclareFontShape{OT2}{cmss}{m}{n}{<->wncyss10}{}

\DeclareRobustCommand\cyr{\fontencoding{OT2}\selectfont}
\DeclareTextFontCommand{\textcyr}{\cyr}

\DeclareUnicodeCharacter{0306}{}

\newtheorem{theorem}{Theorem}[section]
\newtheorem{prop}[theorem]{Proposition}

\newtheorem{corr}[theorem]{Corollary}

\theoremstyle{definition}

\newtheorem{deff}[theorem]{Definition}

\theoremstyle{remark}
\newtheorem{comm}[theorem]{Remark}

\newcommand{\R}{\mathbb R}

\newcommand{\C}{\mathbb C}

\newcommand{\p}{\partial}

\newcommand{\z}{\bar z}

\newcommand{\dbar}{\bar\partial}

\newcommand{\bc}{\mathbb{B}}
\newcommand{\holb}{Hol(D,\bc)}
\newcommand{\hpb}{H^p(D,\bc)}
\newcommand{\hab}{H_{A,B}(D,\bc)}
\newcommand{\hpab}{H^p_{A,B}(D,\bc)}
\newcommand{\hw}{H_w(D, \bc)}

\newcommand{\sca}{\text{Sc}\,}
\newcommand{\vect}{\text{Vec}\,}

\title{Bicomplex Hardy Classes of Solutions to Higher-Order Vekua
Equations}

\author{William L. Blair}
\address{Department of Mathematics\\
  The University of Texas at Tyler\\
  Tyler, TX 75799}
\email{wblair@uttyler.edu}



\begin{document}

\begin{abstract}
    We extend representation formulas for functions in Hardy classes of solutions to higher-order iterated Vekua equations to Hardy classes of bicomplex-valued functions that solve a bicomplex version of the Vekua equation or its higher-order generalizations. Using these representations, we show that functions in these bicomplex-valued Hardy classes have nontangential boundary values and boundary values in the sense of distributions. 
\end{abstract}

\maketitle

\section{Introduction}

In this paper, we extend the theory of Hardy spaces to solutions of the bicomplex Vekua equation.

The Vekua equation is a well-studied first-order complex partial differential equation that generalizes the classical Cauchy-Riemann equation. It is an example of a nonhomogeneous Cauchy-Riemann equation, i.e., right-hand side of the equation is not identically equal to zero, where in this specific case, the right-hand side is a linear combination of the solution function and its complex conjugate with coefficients that are complex-valued functions (usually elements of a Lebesgue space). The Vekua equation is interesting because it is the model case for first-order elliptic partial differential equations in planar domains and arises in the study of infinitesimal bending of surfaces. The solutions of the Vekua equation are called generalized analytic functions (as named by I. N. Vekua \cite{Vek}) and pseudoanalytic functions (as named by L. Bers \cite{Bers}). From the point of view of complex function theory, these functions are a fascinating generalization of holomorphic functions. It is well known that on bounded domains generalized analytic functions have the form of a product of a holomorphic function and a function bounded in modulus away from zero. Consequently, many properties of holomorphic functions, in particular those that reply on size arguments or behavior of zero sets, immediately extend to generalized analytic functions. This is clear when one considers boundary behavior. 

The study of holomorphic functions on planar domains with boundary values centers on the study of Hardy spaces. The classical holomorphic Hardy spaces of functions on the complex unit disk are the holomorphic functions with bounded mean modulus, raised to a positive real power, on concentric circles centered at the origin and contained in the unit disk. It is a classical result that those functions converge along paths contained in nontangential approach regions to Lebesgue space functions on the unit circle, see \cite{Duren, Koosis, Rep, RudinCR, BAF}, and the functions converge to these boundary values in the corresponding $L^p$ norm. Also, since the Lebesgue spaces, for small Lebesgue exponent $p < 1$, are not as useful as those with exponent $p \geq 1$, it is desirable to have a replacement for nontangential boundary values in the small exponent case. One generalization of the boundary value for functions in the disk is the boundary value in the sense of distributions (or distributional boundary value). The holomorphic Hardy space functions were shown to have distributional boundary values by G. Hoepfner and J. Hounie in \cite{GHJH2} (see also E. Straube's paper \cite{Straube}).

In \cite{KlimBook}, S. Klimentov proved that generalized analytic functions which satisfy the same growth condition imposed on functions in the classic holomorphic Hardy spaces also have nontangential boundary values that are $L^p$ functions on the circle. Also, S. Berhanu and J. Hounie, in \cite{BerHou}, showed that generalized analytic functions with distributional boundary values that are in $L^p$ on the circle are elements of the Vekua-Hardy spaces as considered by Klimentov. In \cite{WBD}, the author and B. B. Delgado showed that this behavior characterizes the functions in the Vekua-Hardy spaces. 

Another natural extension of holomorphic functions are the polyanalytic functions. These functions are the polynomials in the complex conjugate of a complex variable with holomorphic coefficients, and they solve the higher-order generalization of the Cauchy-Riemann equation where the Cauchy-Riemann operator is iterated one more time than the degree of the polynomial. In \cite{polyhardy}, polyanalytic functions on the unit disk satisfying a higher-order analogue of the Hardy space size condition are shown to inherit the boundary behavior of the classic Hardy spaces, specifically they have nontangential $L^p$ boundary values and the functions converge to those boundary values in the $L^p$ norm. 

Motivated by the results above (and improvements in \cite{metahardy, WB3}), the author and B. B. Delgado in \cite{WBD} proved the Hardy space boundary behavior extends to solutions of higher-order iterated Vekua equations, i.e., solutions of a differential equation similar to the one solved by polyanalytic functions where the Cauchy-Riemann operator is replaced by one corresponding to the Vekua equation have $L^p$ nontangential boundary values, the functions converge to those boundary values in the associated $L^p$ norm, and these functions have distributional boundary values. This generalized, combined, and extended the work found in \cite{KlimBook, BerHou, polyhardy, metahardy, WB3}. Functions solving higher-order iterated Vekua equations were considered previously in \cite{itvek, itvekbvp, Pascali}. 

In the continuing pursuit of generalizations of the complex numbers, the bicomplex numbers have recently produced much interest. The bicomplex numbers are essentially $\mathbb{C}^2$ with a commutative multiplication, see \cite{PriceMulti}, that produces desirable algebraic properties. One direction where bicomplex numbers naturally arise is in the study of the stationary Schr\"odinger equation. The differential operator associated with the stationary Schr\"odinger equation may be factorized into two first-order differential operators where one of them is a bicomplex version of the operator associated with the first-order Vekua equation, see \cite{ComplexSchr, CastaKrav, FundBicomplex}. This realization stimulates an interest in understanding solutions to this bicomplex Vekua-type equation. Note, the functions under consideration here are functions of a complex variable taking values in the bicomplex numbers. Recently, Bergman spaces of solutions to the bicomplex Vekua equation that are also in Lebesgue spaces defined for bicomplex-valued functions have been considered by V. Vicente-Ben\'itez in \cite{BCTransmutation,BCBergman}. Similar spaces in a Clifford analysis setting are considered by B. B. Delgado in \cite{Hodge} and for functions of a bicomplex variable in \cite{BCBergmanBloch, BCWeightedBergman}.

We work to develop Hardy classes for the solutions of the bicomplex Vekua equation considered in \cite{BCTransmutation,BCBergman}. Also, we develop Hardy classes for the higher-order generalizations of these Vekua equations of the form considered in \cite{WBD}. After defining these Hardy classes of solutions to the bicomplex Vekua equation and its higher-order variants, we show that functions in these classes exhibit the boundary behavior expected of functions from a Hardy class, e.g., the functions have nontangential boundary values, the functions converge to these nontangential boundary values in the associated $L^p$ norm, and these functions have distributional boundary values. This directly extends the results from \cite{WBD} into the bicomplex setting and furthers the study of generalizations of the Hardy spaces. Note, Hardy classes of bicomplex-valued functions of a single complex variable were previously considered in \cite{BCAtomic}. 

We provide the outline for the paper. In Section \ref{background}, we provide background for holomorphic Hardy spaces, solutions of nonhomogeneous Cauchy-Riemann equations (including the Vekua equation), generalizations of the Hardy spaces associated with solutions of nonhomogeneous Cauchy-Riemann equations (in particular previous results for the Vekua-Hardy spaces), and bicomplex numbers necessary for the remainder of the paper. In Section \ref{bcholohardyspaces}, we provide background concerning the bicomplex-holomorphic Hardy spaces. These were previously considered in \cite{BCAtomic}, and  for completeness, we include the results from that paper which will be referenced. We also prove a new theorem for the bicomplex-holomorphic Hardy spaces that generalizes a theorem from \cite{GHJH2} that relates boundary values in the sense of distributions, Poisson-type integral representation, and growth near the boundary. In Section \ref{BCVekuaHardyspaces}, the bicomplex-Vekua-Hardy spaces are defined, representation and inclusion results are proved and are used, along with results about the generalized complex-valued Hardy class functions considered in \cite{WB, BCAtomic}, to show that functions in the bicomplex-Vekua-Hardy spaces have nontangential boundary values in $L^p$ and converge to their nontangential boundary values in the associated $L^p$ norm. Also, certain conditions are shown to be sufficient for solutions to bicomplex-Vekua equations to have distributional boundary values. In Section \ref{BCPolyHardyspaces}, bicomplex-valued polyanalytic functions are considered along with their corresponding Hardy spaces. In particular, representation and inclusion results mirroring those of complex-valued polyanalytic functions are proved and are used to prove representation theorems for the associated poly-Hardy spaces that extend similar results from the complex-valued setting of \cite{polyhardy}. Also, existence of $L^p$ nontangential boundary values and convergence in the $L^p$ norm is shown. In Section \ref{BCHOIVHardyspaces}, we define the bicomplex-HOIV Vekua equations, their solutions, and the associated Hardy classes. We prove representation results for these functions that directly extend those from the complex-setting in \cite{WBD}, existence of their $L^p$ nontangential boundary values, and that the functions converge to these nontangential boundary values in the associated $L^p$ norm, as well as sufficient conditions for these functions to have distributional boundary values.

\section{Background}\label{background}

We begin this section with a brief description of notation. We work exclusively on the domain $D$ which we take to be the disk of radius one centered at the origin in the complex plane. By $L^p(S)$, we denote the complex-valued functions defined on a set $S$ with integrable modulus raised to the $p^{\text{th}}$-power. We represent the space of distributions on the boundary of the unit disk $\p D$ by $\mathcal{D}'(\p D)$. Finally, in an effort to disambiguate between other notions of conjugation that will arise below, we denote the complex conjugate of $z \in \C$ by $z^*$, i.e., if $z = x + iy$, where $x, y \in \R$, then $z^* = x - iy$. 

\subsection{Holomorphic Hardy Spaces}\label{holohardy}

We recall facts about the classic holomorphic Hardy spaces that we generalize. 

\begin{deff}
    We define $Hol(D)$ to be the space of functions $f: D \to \C$ such that 
    \[
        \frac{\p f}{\p z^*} = 0. 
    \]
\end{deff}

\begin{deff}
    For $0 < p < \infty$, we define $H^p(D)$ to be the space of functions $f \in Hol(D)$ such that 
    \[
        ||f||_{H^p} := \sup_{0 < r < 1} \left( \int_0^{2\pi} |f(re^{i\theta})|^p \,d\theta \right)^{1/p} < \infty.
    \]
\end{deff}

\begin{theorem}[\cite{Duren, Koosis, Rep, RudinCR, BAF}]\label{bvcon}
A function $f \in H^p(D)$, $0 < p < \infty$, has nontangential boundary values $f_{nt} \in L^p(\partial D)$ at almost every point of $\p D$, 
\[
\lim_{r\nearrow 1} \int_0^{2\pi} |f(re^{i\theta})|^p \, d\theta = \int_0^{2\pi} |f_{nt}(e^{i\theta})|^p \,d\theta,
\]
and
\[
\lim_{r\nearrow 1} \int_0^{2\pi} |f(re^{i\theta})- f_{nt}(e^{i\theta})|^p \, d\theta = 0.
\]
\end{theorem}

As the Lebesgue spaces for $p$-values satisfying $0 < p < 1$ are not nearly so useful as those satisfying $p \geq 1$, we consider a more general kind of boundary value for the functions in $H^p(D)$ spaces. We define the boundary value in the sense of distributions and recall, after the definition, results of Hoepfner and Hounie that connect distributional boundary values and the holomorphic Hardy spaces.

\begin{deff}\label{bvcircle}Let $f: D \to \C$. We say that $f$ has a boundary value in the sense of distributions, denoted by $f_b \in \mathcal{D}'(\p D)$, if, for every $\varphi \in C^\infty(\partial D)$, the limit
            \[
             \langle f_b, \varphi \rangle := \lim_{r \nearrow 1} \int_0^{2\pi} f(re^{i\theta}) \, \varphi(\theta) \,d\theta
            \]
            exists.
            
\end{deff}

\begin{theorem}[Theorem 3.1 \cite{GHJH2}]\label{GHJH23point1}
For $f \in Hol(D)$, the following are equivalent:
\begin{enumerate}
    \item For every $\phi \in C^\infty(\p D)$, there exists the limit
    \[
             \langle f_b, \phi \rangle := \lim_{r \nearrow 1} \int_0^{2\pi} f(re^{i\theta}) \, \phi(\theta) \,d\theta.
            \]

    \item There is a distribution $f_b \in \mathcal{D}'(\p D)$ such that $f$ is the Poisson integral of $f_b$
    \[
             f(re^{i\theta}) = \frac{1}{2\pi}\langle f_b, P_r(\theta - \cdot) \rangle,
            \]
    where 
    \[
        P_r(\theta) = \frac{1-r^2}{1-2r\cos(\theta) +r^2}
    \]
    is the Poisson kernel on $D$. 

    \item There are constants $C>0$, $\alpha \geq 0$, such that 
    \[
        |f(re^{i\theta})| \leq \frac{C}{(1-r)^\alpha},
    \]
    for $0 \leq r < 1$.
\end{enumerate}
\end{theorem}

\begin{theorem}[Corollary 3.1 \cite{GHJH2}]\label{GHJH23point1corr}
The functions in $H^p(D)$, $0 < p < \infty$, satisfy (3) in Theorem \ref{GHJH23point1}
\end{theorem}

We note here that distributional boundary values of holomorphic (and harmonic) functions were also considered previously in \cite{Straube}.

Next, we describe an integral operator and include associated results developed by Vekua \cite{Vek} to study solutions of nonhomogeneous Cauchy-Riemann equations in the complex setting. 

\begin{deff}\label{vekopondisk}
For $f: D \to \mathbb{C}$ and $z  \in \mathbb{C}$, we denote by $T(\cdot)$ the integral operator defined by 
\[
T(f)(z) = -\frac{1}{\pi} \iint_D \frac{f(\zeta)}{\zeta - z}\, d\xi\,d\eta,
\]
whenever the integral is defined.
\end{deff}

\begin{theorem}[\cite{BegBook}, Theorems 1.19 and 1.26 \cite{Vek}, Theorems 1.4.7 and 1.4.8 \cite{KlimBook}]\label{VekKlimThmcombined}
For $f \in L^1(D)$, 
\[
    \frac{\p}{\p z^*} T(f) = f.
\]
If $f \in L^q(D)$, $1 \leq q \leq 2$, then $T(f) \in L^\gamma(D)$, where $1 < \gamma < \frac{2q}{2-q}$, $T(f)|_{\p D(0,r)} \in L^\mu(\p D(0,r))$, 
\[
||Tf||_{L^\mu(\p D(0,r))} \leq C ||f||_{L^q(D)},
\]
$0 < r \leq 1$, and 
\[
\lim_{r\nearrow 1} \int_0^{2\pi} |T(f)(e^{i\theta}) - T(f)(re^{i\theta})|^\mu \,d\theta = 0, 
\]
where $\mu$ satisfies $1 < \mu < \frac{q}{2-q}$ and $C$ is a constant that does not depend on $r$ or $f$. For $f \in L^q(D)$, $q > 2$, $T(f) \in C^{0,\alpha}(\overline{D})$, with $\alpha = \frac{ q-2}{q}$. 
\end{theorem}

\subsection{Bicomplex Numbers}

We now work to describe the bicomplex numbers and provide sufficient background for the results in the proceeding sections. We use the notation and terminology from \cite{BCTransmutation, BCBergman, BCAtomic}. For additional material and background on the bicomplex numbers and functions taking bicomplex values, see \cite{FundBicomplex, BicomplexHilbert, KravAPFT, BCHolo, ComplexSchr}.

The bicomplex numbers $\bc$, as a set, can be described as $\C^2$. Letting $\bc$ inherit the usual addition and multiplication by complex scalars, then with the multiplication defined by 
\[
    (z_1, z_2)(w_1,w_2) = (z_1w_1 - z_2w_2, z_1w_2 + z_2 w_1),
\]
for any $(z_1,z_2), (w_1,w_2) \in \C^2$, $\bc$ is a commutative algebra over $\C$. Identifying elements of the form $(a,0)$ with the respective complex number $a$ and defining $j = (0,1)$, we can represent any $(z_1, z_2) \in \bc$ with $z_1 + j z_2$, where $j^2 = -1$ and $ij = ji$. 
Note, the elements 
\[
    p^{\pm} := \frac{1}{2}(1\pm ij),
\]
are idempotent and satisfy
\[
    p^+ p^- = 0.
\]
So, $\bc$ has nonzero zero divisors.

\begin{deff}
    We define the scalar part of $z = z_1 + jz_2 \in \bc$ as 
    \[
        \sca z = z_1,
    \]
    and the vector part of $z = z_1 + jz_2 \in \bc$ as 
    \[
        \vect z = z_2.
    \]
\end{deff}

\begin{deff}
    We define the bicomplex conjugate of $z = z_1 + jz_2 \in \bc$ as 
    \[
        \overline{z} := z_1 - jz_2.
    \]
\end{deff}
 
    Since there is the opportunity for confusion between complex number conjugation and bicomplex number conjugation, we denote complex conjugation for $u \in \mathbb{C}$ by $u^*$. Hence, complex-valued holomorphic functions will be those $f: D \to \C$ such that 
    \[
        \frac{\p f}{\p z^*} = 0.
    \]

\begin{prop}[Proposition 1 \cite{BCTransmutation}]\label{everybchasplusandminus}
    Let $w \in \bc$. There exist unique $w^{\pm} \in \C$ such that 
    \[
        w = p^+ w^+ + p^- w^-.
    \]
    Furthermore, 
    \[
    w^{\pm} = \sca w \mp i \vect w.
    \]
\end{prop}

\begin{deff}
    For $w \in \bc$, we define the bicomplex norm of $w$, denoted by $||\cdot||_{\bc}$,  as
    \[
        ||w||_{\bc} := \sqrt{\frac{|w^+|^2 + |w^-|^2}{2}},
    \]
    where $|w^{\pm}|$ is the usual complex modulus. 
\end{deff}

It is immediate from the definition of $||\cdot||_{\bc}$ that for $w = p^+ w^+ + p^- w^- \in \bc$, 
    \begin{equation}\label{bcbasicestimates}
        \frac{1}{\sqrt{2}} |w^\pm| \leq ||w||_{\bc} \leq \frac{1}{\sqrt{2}}\left( |w^+| + |w^-|\right).
    \end{equation}
Also, for $w,v \in \bc$, we have
    \begin{equation}\label{stareqn}
        ||wv||_\bc \leq \sqrt{2} \, ||w||_\bc \, ||v||_\bc.
    \end{equation}
We use these basic estimates throughout. Next, we define the bicomplex Lebesgue spaces that we use and recall a proposition that relates them to the familiar complex-valued Lebesgue spaces.

\begin{deff}
    For $0 < p < \infty$, we define the bicomplex Lebesgue space $L^p(D, \bc)$ to be the set of functions $f: D \to \bc$ such that 
    \[  
        ||f||_{L^p(D,\bc)}:= \left( \iint_D ||f(z)||^p_{\bc} \,dx\,dy \right)^{1/p} < \infty.
    \]
    Similarly, we define $L^p(\p D, \bc)$ to be the set of functions $g: \p D \to \C$ such that 
    \[
        ||g||_{L^p(\p D,\bc)}:= \left( \int_{\p D} ||g(z)||^p_{\bc} \,d\sigma(z) \right)^{1/p} < \infty,
    \]
    where $\sigma$ is Lebesgue surface measure on $\p D$. 
\end{deff}

\begin{prop}[Proposition 2.26 \cite{BCAtomic}]\label{propLqiff}
    For $p$ a positive real number, a function $f = p^+ f^+ + p^- f^- \in L^p(D, \bc)$ if and only if $f^+, f^- \in L^p(D)$. The analogous result holds for $f \in L^p(\p D, \bc)$.
\end{prop}

Next, we define the bicomplex differential operators that we use to construct the bicomplex Vekua-type equations that we consider. Solutions of differential equations associated with these differential operators were previously considered in, for example,  \cite{BCTransmutation, BCBergman, BCAtomic}.

\begin{deff}\label{bcdbardef}
    We define the bicomplex differential operators $\p$ and $\dbar$ as 
    \[
        \p := \frac{1}{2} \left( \frac{\p}{\p x} - j \frac{\p }{\p y}\right)
    \]
    and
    \[
        \dbar := \frac{1}{2} \left( \frac{\p}{\p x} + j \frac{\p }{\p y}\right).
    \]
\end{deff}

   It is immediate from the definition of $\p$ and $\dbar$ that 
    \[
        \p:= p^+ \frac{\p}{\p z^*} + p^- \frac{\p }{\p z}
    \]
    and 
    \[
        \dbar := p^+ \frac{\p}{\p z} + p^- \frac{\p }{\p z^*},
    \]
    where $\frac{\p}{\p z}$ and $\frac{\p }{\p z^*}$ are the usual complex partial differential operators. Using the idempotent representation for bicomplex numbers, if $w = p^+ w^+ + p^- w^-$, then $\dbar \, w = 0$ if and only if 
    \[
        \frac{\p w^+}{\p z} = 0 = \frac{\p w^-}{\p z^*}, 
    \]
    i.e., $(w^+)^*,w^-$ are complex-valued holomorphic functions. 

We require an operation and associated symbol that is specific to bicomplex numbers and their association with the complex numbers.

\begin{deff}
    We define the bicomplexification $\widehat{u}$ of a complex number $u = x + iy$ by 
    \[
        \widehat{u} = x + jy. 
    \]
\end{deff}

Using bicomplexification, we are able to transform a complex number into a bicomplex number, or equivalently, we are able to consider the bicomplex numbers with real-valued scalar and vector parts. Observe that, for a complex variable $u$, we have
\begin{align*}
    \dbar \,\widehat{u} = 0,
\quad
    \p \, \widehat{u}^n = n \widehat{u}^{n-1},
\quad
    \p \, \widehat{u^*} = 0,
\quad\text{ and }\quad
    \dbar \,\widehat{u^*}^{n} = n \widehat{u^*}^{n-1}.
\end{align*}
So, polynomials in the bicomplexification of the complex conjugate of a complex variable play the same role with respect to $\dbar$ as polynomials in the complex conjugation of a complex variable with respect to $\frac{\p}{\p z^*}$. 

By appealing to the work of Vekua on the complex $T(\cdot)$, we realize a bicomplex-valued right-inverse operator to $\dbar$. This operator was introduced in \cite{FundBicomplex}. 

\begin{theorem}[\cite{Vek, FundBicomplex, BCTransmutation, BCBergman, conjbel}]\label{Thm: bctoperator}
    For $f \in L^p(D, \bc)$, $p \geq 1$, the bicomplex Theodorescu operator $T_{\bc}(\cdot)$ defined by 
    \[
        T_{\bc} (f)(z) := p^+\left(-\frac{1}{\pi} \iint_D \frac{f^+(\zeta)}{\zeta^* - z^*} \,d\eta\,d\xi \right) + p^-\left(- \frac{1}{\pi}\iint_{D} \frac{ f^-(\zeta)}{\zeta- z}\,d\eta\,d\xi \right)
    \]
    exists and 
    \[
        \dbar \, T_{\bc} (f) = f. 
    \]
\end{theorem}

The next theorem shows that much of the behavior of $T(\cdot)$ is recovered by $T_\bc(\cdot)$. This theorem generalizes Theorem \ref{VekKlimThmcombined} from the complex setting to the setting of bicomplex valued functions and was proved in \cite{BCAtomic}. 

\begin{theorem}[Theorem 2.31 \cite{BCAtomic}]\label{bctbehavior}
For every $f \in L^q(D, \bc)$, $1 \leq q \leq 2$, $T_\bc(f) \in L^\gamma(D,\bc)$, $1 < \gamma < \frac{2q}{2-q}$, $T_\bc(f)|_{\p D(0,r)} \in L^\gamma(\p D(0,r), \bc)$, $0 < r \leq 1$, where $\gamma$ satisfies $1 < \gamma < \frac{q}{2-q}$,  
\[
||T_\bc(f)||_{L^\gamma(\p D(0,r),\bc)} \leq C ||f||_{L^q(D,\bc)},
\]
where $C$ is a constant that does not depend on $r$ or $f$, and
\[
\lim_{r\nearrow 1} \int_0^{2\pi} ||T_\bc(f)(e^{i\theta}) - T_\bc(f)(re^{i\theta})||_\bc^\gamma \,d\theta = 0, 
\]
for $1 \leq \gamma < \frac{q}{2-q}$. For $f \in L^q(D,\bc)$, $q > 2$, $T(f) \in C^{0,\alpha}(\overline{D},\bc)$, with $\alpha = \frac{ q-2}{q}$. 
\end{theorem}

The next definition extends Definition \ref{bvcircle} to the bicomplex setting. We follow this definition with a useful theorem that relates a function's inclusion in certain Lebesgue spaces with existence of a distributional boundary value.

\begin{deff}\label{bcbvcircle}Let $f: D \to \bc$. We say that $f$ has a boundary value in the sense of distributions, denoted by $f_b \in \mathcal{D}'(\p D)$, if, for every $\varphi \in C^\infty(\partial D)$, the limit
            \[
             \langle f_b, \varphi \rangle := \lim_{r \nearrow 1} \int_0^{2\pi} f(re^{i\theta}) \, \varphi(\theta) \,d\theta
            \]
            exists.
            
\end{deff}

\begin{theorem}[Theorem 2.30 \cite{BCAtomic}]\label{lonedistbv}
For every $f \in L^1(D, \bc)$ such that $f$ has a $L^1(\p D,\bc)$ boundary value and 
\[
    \lim_{r \nearrow 1} \int_0^{2\pi} ||f(re^{i\theta}) - f(e^{i\theta})||_\bc \,d\theta = 0,
\]
$f_b$ exists and $f_b = f|_{\p D}$ as distributions. 
\end{theorem}

\subsection{Generalizations of the Hardy Spaces}

\color{black}

The study of classes of complex-valued functions defined on the disk that generalize the classic holomorphic Hardy spaces is an active area of function theory. One direction experiencing significant success is the study of solutions to certain differential equations, usually generalizations of the Cauchy-Riemann equation in some way, that also have finite $H^p$ norm. See, for example, \cite{KlimBook, PozHardy, CompOp, moreVekHardy, conjbel, threedvek, quasiHardy, quasiHardy2, quasiHardy3, polyhardy, metahardy, WB3, WBD, WB}.

The goal of this paper is to extend results found in \cite{WB3, WBD} concerning representation and boundary behavior of functions in Hardy-type classes of solutions to the complex Vekua equation
\[
    \frac{\p w}{\p z^*} = Aw + Bw^*,
\]
 and higher-order variations, to the setting of Hardy-type classes of functions defined on $D$ that take values in $\bc$ and solve bicomplex Vekua-type equations
\[
\dbar w = Aw + B\overline{w}
\]
and their associated higher-order variants. For completeness, we recall some relevant definitions and results. 

\begin{deff}\label{complexnonhomoghardydeff}
For $0 < p < \infty$ and $f: D \to \C$, we define the Hardy classes $H^p_f(D)$ to be the collection of functions $w: D \to \C$ such that 
\[
    \frac{\p w}{\p z^*} = f
\]
and 
\[
    \sup_{0 < r < 1} \int_0^{2\pi} |w(re^{i\theta})|^p \,d\theta < \infty.
\]
\end{deff}

\begin{comm}
    For $f \equiv 0$, $H^p_f(D) = H^p_0(D)$ is exactly the classic holomorphic Hardy space $H^p(D)$. 
\end{comm}

The boundary behavior of functions in the Hardy classes $H^p_f(D)$ was explored in \cite{WB,BCAtomic}. The following theorem shows that, under the right conditions, these functions have boundary values in the same way as the classical Hardy spaces. 

\begin{theorem}[Theorem 2.40 \cite{BCAtomic}]\label{thm: nonhomogHpbvcon}
For $0 < p < \infty$ and $q>2$, every $w \in H^p_f(D)$, where $f \in L^q(D)$, has a nontangential boundary value $w_{nt} \in L^p(\p D)$ and 
\[
\lim_{r \nearrow 1} \int_0^{2\pi} |w_{nt}(e^{i\theta}) - w(re^{i\theta}) |^p \, d\theta  = 0.
\]
The result holds when $1 < q \leq 2$, so long as $p$ satisfies $p < \frac{q}{2-q}$. 
\end{theorem}

A special case, of significant interest, of the above Hardy classes of solutions to nonhomogeneous Cauchy-Riemann equations is the Vekua-Hardy spaces defined below.

\begin{deff}
    For $0 < p < \infty$ and $A, B : D\to\mathbb{C}$, we define the Vekua-Hardy spaces $H^p_{A,B}(D)$ to be the collection of $w: D \to \mathbb{C}$ such that 
    \begin{equation}\label{complexVekua}
    \frac{\p w}{\p z^*} = Aw + Bw^*
\end{equation}
and 
\[
    \sup_{0 < r < 1} \int_0^{2\pi} |w(re^{i\theta})|^p \,d\theta < \infty.
\]
\end{deff}

\begin{comm}Functions that solve (\ref{complexVekua}) are called generalized analytic functions by those who follow the works of Vekua \cite{Vek} and are called pseudoanalytic functions (or pseudoholomorphic functions) by those who follow the works of Bers \cite{Bers}. For a specific pair of coefficients $A$ and $B$, we denote by $H_{A,B}(D)$ the collection of solutions of (\ref{complexVekua}). The Vekua-Hardy spaces (and their higher-order variations) have been previously considered in, for example, \cite{KlimBook, PozHardy, CompOp, moreVekHardy, conjbel,threedvek, WB3, WBD}. \end{comm}

Some relevant results that will be referenced throughout the paper are the following.

\begin{theorem}[``The Basic Lemma'' \cite{Vek}, Theorems 2.1.1, 2.1.3, 2.1.5, 2.1.6 \cite{KlimBook}, Corollary 5.6 \cite{WBD}] \label{VekHardyCombined}
Let $A, B \in L^q(D)$, $q>2$, and $0 < p < \infty$. 
    \begin{enumerate}
    \item Every solution $w: D \to \mathbb{C}$ of the Vekua equation
    \[
        \frac{\p w}{\p z^*} = Aw + Bw^*
    \] has the form $w = \varphi e^\phi$, where $\varphi \in Hol(D)$ and $\phi \in C^{0, \alpha}(\overline{D})$, $\alpha = \frac{q-2}{q}$. 
    \item A function $w = \varphi e^\phi$ is in $H^p_{A,B}(D)$ if and only if $\varphi \in H^p(D)$. 
    \item A function $w \in H^p_{A,B}(D)$ has a nontangential boundary value $w_{nt} \in L^p(\p D)$ almost everywhere on $\p D$, 
    \begin{equation}\label{equationtwopointone}
    \lim_{r \nearrow 1} \int_0^{2\pi} |w(re^{i\theta})|^p \,d\theta = \int_0^{2\pi} |w_{nt}(e^{i\theta})|^p \,d\theta,
    \end{equation}
    and
    \begin{equation}\label{equationtwopointtwo}
    \lim_{r\nearrow 1}\int_0^{2\pi} |w(re^{i\theta}) - w_{nt}(e^{i\theta})|^p \,d\theta = 0.
    \end{equation}
     If $w_{nt} \in L^s( \p D)$, $s>p$, then $w \in H^s_{A,B}(D)$. 
    \item Every function $w \in H^p_{A,B}(D)$ is an element of $L^m(D)$, for every $0 < m < 2p$. 
     \item For $p \geq 1$, $w \in H^p_{A,B}(D)$ if and only if $w \in H_{A,B}(D)$, $w_b$ exists, and $w_b \in L^p(\p D)$.
    \end{enumerate}
\end{theorem}

\section{Bicomplex Hardy Spaces}\label{bcholohardyspaces}

We begin by defining the bicomplex-valued Hardy spaces of functions defined on the complex unit disk. These classes of functions were previously considered in \cite{BCAtomic}.

\begin{deff}
We define the $\bc$-holomorphic functions on $D$ to be the functions $f: D \to \bc$ such that 
    \[
        \dbar f = 0.
    \]
    We denote by $\holb$ the space of all $\bc$-holomorphic functions on $D$.  Similarly, we define the $\bc$-anti-holomorphic functions on $D$ to be the functions $f: D \to \bc$ such that 
    \[
        \p f = 0.
    \]
    We denote by $\overline{\holb}$ the space of all $\bc$-anti-holomorphic functions on $D$.
\end{deff}

\begin{deff}
Let $w: D \to \bc$. We define $\hw$ to be those functions $f: D \to \bc$ such that 
\[
    \dbar f = w.
\]
\end{deff}

\begin{deff}
    For $0 < p < \infty$ and a function $f: D \to \bc$, we define the bicomplex-$H^p$ norm to be
    \[
        ||f||_{H^p_{\bc}} := \sup_{0 < r < 1}\left( \int_0^{2\pi} ||f(re^{i\theta})||^p_{\bc} \,d\theta\right)^{1/p}.
    \]
\end{deff}

\begin{comm}
    The bicomplex-$H^p$ norm is only truly a norm when $p \geq 1$ and is a quasi-norm for $0 < p < 1$. Following the convention in the literature for analogous objects, we call this quantity a norm regardless of the value of $p$.  
\end{comm}

\begin{deff}
    Let $0 < p < \infty$. We define the bicomplex Hardy spaces $\hpb$ to be those functions $f \in \holb$ such that 
    $||f||_{H^p_{\bc}} < \infty$.
\end{deff}

In \cite{BCAtomic}, the following three results were proved for the $\bc$-holomorphic Hardy spaces $H^p(D,\bc)$. We recall them here to familiarize the reader with these spaces, and we prove an additional inclusion that will be extended later.

\begin{theorem}[Theorem 4.1 \cite{BCAtomic}]\label{thm: bchardyrep}
    For $0 < p < \infty$, a function $f = p^+f^+ + p^-f^- \in \hpb$ if and only if $(f^+)^*,f^- \in H^p(D)$. 
\end{theorem}

\begin{theorem}[Theorem 4.3 \cite{BCAtomic}]\label{bcholobvcon}
    For $0 < p < \infty$, every $f = p^+ f^+ + p^- f^-\in \hpb$ such that $f^+ \in H^p_w(D)$, where $w \in L^q(D)$, and $q >2$ or $1 < q \leq 2$ and $p$ satisfies $p < \frac{q}{2-q}$, has a nontangential limit $f_{nt} \in L^p(\p D, \bc)$, and 
    \[
        \lim_{r \nearrow 1} \int_0^{2\pi} ||f_{nt}(e^{i\theta}) - f(re^{i\theta}) ||^p_{\bc} \, d\theta  = 0. 
    \] 
\end{theorem}

\begin{prop}[Proposition 7.2 \cite{BCAtomic}]\label{Prop: bchardymlessthan2p}
For $0 < p < \infty$, every $w \in H^p(D,\bc)$ is an element of $L^m(D,\bc)$, for all $0 < m < 2p$. 
\end{prop}

Now, we show that if the nontangential boundary value of a function in $H^p(D,\bc)$ is a member of a more restricted Lebesgue space, then this implies inclusion in a more restricted Hardy space for the function. This extends a theorem of Smirnov, see Theorem 1.28 \cite{fcotud}.

\begin{theorem}\label{bcbdinsfuncins}
Let $p$ be a positive real number. Every $f \in H^p(D,\bc)$ such that $f_{nt} \in L^s(\p D, \bc)$, where $s > p$, is an element of $H^s(D,\bc)$. 
\end{theorem}

\begin{proof}
By Theorem \ref{thm: bchardyrep}, if $f \in H^p(D,\bc)$, then $(f^+)^*, f^- \in H^p(D)$. By Theorem \ref{propLqiff}, since $f_{nt} \in L^s(\p D,\bc)$, it follows that $f^\pm_{nt} \in L^s(\p D)$. Since 
\[
    |(f^+)^*(z)| = |f^+(z)|
\]
for all $z \in D$, it follows that $(f^+)^*_{nt}, f^-_{nt} \in L^s(\p D)$. By Theorem \ref{VekHardyCombined} statement (3), $(f^+)^*,f^- \in H^s(D)$. Therefore, by Theorem \ref{thm: bchardyrep}, $f \in H^s(D,\bc)$. 
\end{proof}

We conclude this brief section concerning the $\bc$-holomorphic Hardy spaces by proving a theorem that extends Theorem \ref{GHJH23point1} to this setting. The theorem will be used to show existence of distributional boundary values in the more general settings that follow. 

\begin{theorem}\label{bcGHJH23point1}
For $f \in Hol(D,\bc)$, the following are equivalent:
    \begin{enumerate}
         \item For every $\phi \in C^\infty(\p D)$, there exists the limit
    \[
             \langle f_b, \phi \rangle := \lim_{r \nearrow 1} \int_0^{2\pi} f(re^{i\theta}) \, \phi(\theta) \,d\theta.
            \]

    \item There is a distribution $f_b = p^+ f_b^+ + p^- f^-_b \in \mathcal{D}'(\p D)$ such that $f$ is representable as
    \[
             f(re^{i\theta}) = p^+\left(\frac{1}{2\pi}\langle (f^+_b)^*, P_r(\theta - \cdot) \rangle \right)^* + p^-\frac{1}{2\pi}\langle f^-_b, P_r(\theta - \cdot) \rangle,
            \]
    where 
    \[
        P_r(\theta) = \frac{1-r^2}{1-2r\cos(\theta) +r^2}
    \]
    is the Poisson kernel on $D$. 

    \item There are constants $C>0$, $\alpha \geq 0$, such that 
    \[
        ||f(re^{i\theta})||_\bc \leq \frac{C}{(1-r)^\alpha},
    \]
    for $0 \leq r < 1$.
    \end{enumerate}
\end{theorem}

\begin{proof}
Suppose there exists constants $C>0$, $\alpha \geq 0$, such that 
    \[
        ||f(re^{i\theta})||_\bc \leq \frac{C}{(1-r)^\alpha},
    \]
    for $0 \leq r < 1$. So, by \eqref{stareqn},
    \begin{align*}
        |f^\pm(re^{i\theta})| \leq \sqrt{2}||f(re^{i\theta})||_\bc \leq  \frac{\sqrt{2}C}{(1-r)^\alpha}.
    \end{align*}
    for all $0 \leq r < 1$. Since $f \in Hol(D,\bc)$, it follows by Theorem \ref{thm: bchardyrep} that $(f^+)^*,f^- \in Hol(D)$. Since $|(f^+)^*(re^{i\theta})| = |f^+(re^{i\theta})|$, for all $re^{i\theta}$, it follows, by Theorem \ref{GHJH23point1}, that $(f^+)^*_b, f^-_b$ exist. Since
    \[
       \lim_{r \nearrow 1} \int_0^{2\pi} (f^+)^*(re^{i\theta}) \, \phi(\theta) \,d\theta = \left(\lim_{r \nearrow 1} \int_0^{2\pi} f^+(re^{i\theta}) \, \phi(\theta) \,d\theta ,\right)^*
    \]
    for $\phi \in C^\infty(\p D,\mathbb{R})$, it follows that $f^+_b$ exists. Thus, for every $\phi \in C^\infty(\p D)$, 
    \[
\lim_{r \nearrow 1} \int_0^{2\pi} f(re^{i\theta}) \, \phi(\theta) \,d\theta = p^+\lim_{r \nearrow 1} \int_0^{2\pi} f^+(re^{i\theta}) \, \phi(\theta) \,d\theta + p^-\lim_{r \nearrow 1} \int_0^{2\pi} f^-(re^{i\theta}) \, \phi(\theta) \,d\theta ,
    \]
    and the sum of limits on the right hand exists. Hence, $f_b$ exists.

    Suppose $f_b$ exists. Since, for every $\phi \in C^\infty(\p D)$, we have
    \[
         \int_0^{2\pi} f(re^{i\theta}) \, \phi(\theta) \,d\theta = p^+ \int_0^{2\pi} f^+(re^{i\theta}) \, \phi(\theta) \,d\theta + p^-\int_0^{2\pi} f^-(re^{i\theta}) \, \phi(\theta) \,d\theta, 
    \]
    for every $r$, and
    \[
        \left|\int_0^{2\pi} f^\pm(re^{i\theta}) \, \phi(\theta) \,d\theta \right| \leq \sqrt{2} \left|\left|\int_0^{2\pi} f(re^{i\theta}) \, \phi(\theta) \,d\theta \right|\right|_\bc,
    \]
    for every $r$, it follows that 
    \[
    \left|\lim_{r\nearrow 1}\int_0^{2\pi} f^\pm(re^{i\theta}) \, \phi(\theta) \,d\theta \right| \leq \sqrt{2} \left|\left|\lim_{r \nearrow 1}\int_0^{2\pi} f(re^{i\theta}) \, \phi(\theta) \,d\theta \right|\right|_\bc,
    \]
    and the limit on the right hand side exists. Thus, $f^\pm_b$ exist. By similar reasoning as above, $(f^+)^*_b$ exists. Since $(f^+)^*, f^- \in Hol(D)$, it follows, by Theorem \ref{GHJH23point1}, that 
    \[
        (f^+)^*(re^{i\theta}) = \frac{1}{2\pi} \langle (f^+)^*_b, P_r(\theta - \cdot) \rangle
    \]
    and 
    \[
        f^-(re^{i\theta}) = \frac{1}{2\pi} \langle f^-_b, P_r(\theta - \cdot) \rangle.
    \]
    Since $f= p^+ f^+ + p^- f^-$, it follows that
    \begin{align*}
        f(re^{i\theta}) = p^+\left( \frac{1}{2\pi} \langle (f^+)^*_b, P_r(\theta - \cdot) \rangle\right)^* + p^- \frac{1}{2\pi} \langle f^-_b, P_r(\theta - \cdot) \rangle.
    \end{align*}

    Finally, suppose that 
    \[
f(re^{i\theta}) = p^+\left( \frac{1}{2\pi} \langle (f^+)^*_b, P_r(\theta - \cdot) \rangle\right)^* + p^- \frac{1}{2\pi} \langle f^-_b, P_r(\theta - \cdot) \rangle.
    \]
    Since $(f^+)^*, f^- \in Hol(D)$ and are representable as distributional pairings against the Poisson integral, it follows, by Theorem \ref{GHJH23point1}, that there exist constants $C_\pm > 0$ and $\alpha_\pm \geq 0$ such that 
    \[
        |(f^+)^*(re^{i\theta})| \leq \frac{C_+}{(1-r)^{\alpha_+}}
    \]
    for every $re^{i\theta} \in D$, 
    and
    \[
|f^-(re^{i\theta})| \leq \frac{C_-}{(1-r)^{\alpha_-}},
    \]
    for every $re^{i\theta} \in D$. Thus, 
    \begin{align*}
        ||f(re^{i\theta})||_{\bc} 
        &\leq 2^{-1/2}\left( |f^+(re^{i\theta})| + |f^-(re^{i\theta})|\right) \\
        &=2^{-1/2}\left( |(f^+)^*(re^{i\theta})| + |f^-(re^{i\theta})|\right) \\
        &\leq 2^{-1/2}\left( \frac{C_+}{(1-r)^{\alpha_+}} + \frac{C_-}{(1-r)^{\alpha_-}}\right)\\
        &\leq \frac{C}{(1-r)^\alpha},
    \end{align*}
    where $C:= 2\max\{2^{-1/2}C_+, 2^{-1/2}C_-\}$ and $\alpha := \max\{\alpha_+, \alpha_-\}$, for every $re^{i\theta} \in D$. 

\end{proof}

\section{Bicomplex Vekua-Hardy Spaces}\label{BCVekuaHardyspaces}

In this section, we define the bicomplex Vekua-Hardy spaces and show these classes of functions recover many of the properties of the complex Vekua-Hardy spaces. 

\subsection{Definition}

\begin{deff}
    For $A, B : D \to \bc$, we define $H_{A,B}(D, \bc)$ to be the collection of functions $f: D \to\bc$ that solve the $\bc$-Vekua equation 
    \[
        \dbar f = Af + B\overline{f}.
    \]
\end{deff}

\begin{deff}
    Let $0 < p < \infty$. We define the bicomplex Hardy spaces $\hpab$ to be those functions $f \in \hab$ such that $||f||_{H^p_{\bc}} < \infty$.
\end{deff}

\subsection{Representation}
With the definitions given above, we show that solutions to the $\bc$-Vekua equation are representable in three different forms. We first show that solutions of the $\bc$-Vekua equation have representations similar to those of solutions of the complex Vekua equation and use them to connect the bicomplex Vekua-Hardy spaces to the bicomplex holomorphic-Hardy spaces. The third representation is specific to solutions of the $\bc$-Vekua equation and inspired by a similar structure for $\bc$-holomorphic functions. 

The next proposition shows that solutions of the $\bc$-Vekua equation are representable by a ``representation of the second kind,'' see \cite{Vek} for the corresponding representation in the complex setting.
\begin{prop}[Proposition 11 \cite{FundBicomplex}, Proposition 13 \cite{BCBergman}, see also \cite{Hodge}]\label{Prop: bcrepsecond}
    Every solution $w: D \to \bc$ of 
    \[
        \dbar w = Aw + B\overline{w},
    \]
    with $\dbar w \in L^1(D, \bc)$, is representable as 
    \[
        w = \varphi + T_\bc(Aw + B\overline{w}),
    \]
    where $\varphi \in \holb$. 
\end{prop}

The next theorem is similar to Theorem 5.3 from \cite{BCAtomic} (which is itself a bicomplex extension of Theorem 2.10 from \cite{WB}). It shows that there is a direct relationship between the $\bc$-holomorphic Hardy spaces and the $\bc$-Vekua-Hardy spaces.

\begin{theorem}\label{bcvekhardyrepgeneral}
    Let $0 < p < \infty$. For every $f \in \hab$ such that $\dbar f= Af + B\overline{f} \in L^q(D, \bc)$, $q>2$, with representation $f = \varphi + T_\bc(Af + B\overline{f})$, $ f \in \hpab$ if and only if $\varphi \in \hpb$. The result holds when $1 < q \leq 2$ so long as $p$ satisfies $p < \frac{q}{2-q}$. 
\end{theorem}

\begin{proof}
By hypothesis, $\dbar f= Af + B\overline{f} \in L^1(D, \bc)$, so by Proposition \ref{Prop: bcrepsecond}, 
\[
        f = \varphi + T_\bc(Af + B\overline{f}).
\]
Suppose that $\varphi \in H^p(D,\bc)$. Then, for $r \in (0,1)$, we have
\begin{align}
 \int_0^{2\pi} ||f(re^{i\theta})||_{\bc}^p \,d\theta 
 &= \int_0^{2\pi} ||\varphi(re^{i\theta}) + T_\bc(Af + B\overline{f})(re^{i\theta})||_{\bc}^p \,d\theta \nonumber \\ 
 &\leq C_p\left( \int_0^{2\pi} ||\varphi(re^{i\theta})||_{\bc}^p + \int_0^{2\pi} ||T_\bc(Af + B\overline{f})(re^{i\theta})||_{\bc}^p \,d\theta \right) \nonumber\\
  &\leq C_p\left( ||\varphi||_{H^p_\bc}^p + \int_0^{2\pi} ||T_\bc(Af + B\overline{f})(re^{i\theta})||_{\bc}^p \,d\theta \right), \label{returnpoitbcvekhardy}
\end{align}
where $C_p$ is a constant that depends on only $p$. By Theorem \ref{bctbehavior}, since $ Af+B\overline{f}\in L^q(D,\bc)$, $q>2$, it follows that $T_\bc(Af+B\overline{f}) \in C^{0,\alpha}(\overline{D},\bc)$. Thus, there exists $M>0$ such that 
\[
        ||T_\bc(Af+B\overline{f})(z)||^p_{\bc} \leq M,
\]
for every $z \in \overline{D}$. Hence, 
\begin{align*}
C_p\left( ||\varphi||_{H^p_\bc}^p + \int_0^{2\pi} ||T_\bc(Af + B\overline{f})(re^{i\theta})||_{\bc}^p \,d\theta \right)
&\leq C_p\left( ||\varphi||_{H^p_\bc}^p + 2\pi M \right)< \infty.
\end{align*}
Since there is no dependence on $r$ in the right hand side of the above inequality, we have 
\begin{align*}
||f||_{H^p_\bc}^p &\leq C_p\left( ||\varphi||_{H^p_\bc}^p + 2\pi M \right)< \infty.
\end{align*}
If $1 < q \leq 2$ and $p$ satisfies $p < \frac{q}{2-q}$, then returning to (\ref{returnpoitbcvekhardy}) we have
\begin{align*}
C_p\left( ||\varphi||_{H^p_\bc}^p + \int_0^{2\pi} ||T_\bc(Af + B\overline{f})(re^{i\theta})||_{\bc}^p \,d\theta \right)
&\leq C_p\left( ||\varphi||_{H^p_\bc}^p + C||T_\bc(Af+B\overline{f})||^p_{L^q(D)} \,d\theta \right) < \infty,
\end{align*}
by Theorem \ref{bctbehavior}. Since the right hand side has no dependence on $r$, we have
\[
||f||_{H^p_\bc}^p \leq C_p\left( ||\varphi||_{H^p_\bc}^p + C||T_\bc(Af+B\overline{f})||^p_{L^q(D)} \,d\theta \right) < \infty.
\]
In either case $f \in H^p_{A,B}(D,\bc)$. 

Now, suppose that $f = \varphi + T_\bc(Af + B\overline{f}) \in H^p_{A,B}(D,\bc)$. Observe that, for $r \in (0,1)$, we have
\begin{align}
    \int_0^{2\pi} ||\varphi(re^{i\theta})||_\bc^p \,d\theta 
    &= \int_0^{2\pi} ||f(re^{i\theta}) - T_\bc(Af+B\overline{f})(re^{i\theta})||_\bc^p \,d\theta \nonumber \\
    &\leq C_p \left( \int_0^{2\pi} ||f(re^{i\theta})||_{\bc}^p \,d\theta + \int_0^{2\pi} || T_\bc(Af+B\overline{f})(re^{i\theta})||_\bc^p \,d\theta\right)  \label{returnpointagain}\\
    &\leq C_p\left( ||f||_{H^p_\bc}^p + M 2\pi\right)< \infty,\nonumber
\end{align}
where $C_p$ is a constant that depends on only $p$ and $M$ is a constant that bounds $||T_\bc(Af + B\overline{f})(z)||_\bc^p$, for all $z \in \overline{D}$. So, 
\[
||\varphi||_{H^p_\bc}^p \leq C_p\left( ||f||_{H^p_\bc}^p + M 2\pi\right)< \infty,
\]
and $\varphi \in H^p(D, \bc)$. If $1 < q \leq 2$ and $p$ satisfies $p < \frac{q}{2-q}$, then returning to (\ref{returnpointagain}) and appealing to Theorem \ref{bctbehavior} once more, we have 
\begin{align*}
&C_p \left( \int_0^{2\pi} ||f(re^{i\theta})||_{\bc}^p \,d\theta + \int_0^{2\pi} || T_\bc(Af+B\overline{f})(re^{i\theta})||_\bc^p \,d\theta\right) \\
&\leq C_p \left( \int_0^{2\pi} ||f(re^{i\theta})||_{\bc}^p \,d\theta + C||T_\bc(Af+B\overline{f})||^p_{L^q(D)}\right)  < \infty.
\end{align*}
Therefore, 
\[
    ||\varphi||^p_{H^p_\bc} \leq C_p \left( \int_0^{2\pi} ||f(re^{i\theta})||_{\bc}^p \,d\theta + C||T_\bc(Af+B\overline{f})||^p_{L^q(D)}\right)  < \infty,
\]
and $\varphi \in H^p(D,\bc)$. 

\end{proof}

For the next result, and some that follow, we specialize to the case that $B \equiv 0$. The next theorem shows that solutions to certain $\bc$-Vekua equations have a ``representation of the first kind'' or ``similarity principle'' representation like their complex counterparts, see Theorem \ref{VekHardyCombined}.

\begin{theorem}\label{thm: bcvekbzerorep}
Let $A \in L^1(D, \bc)$. Every $f \in H_{A,0}(D,\bc)$ is representable as 
\[
    f(z) = \varphi(z)e^{\phi(z)},
\]
where $\varphi \in \holb$ and 
$
    \phi(z) = T_\bc(A)(z)$.
In the case that $A \in L^q(D,\bc)$, where $q>2$, $T_\bc(A) \in C^{0,\alpha}(\overline{D},\bc)$, for $\alpha = \frac{q-2}{q}$. 
\end{theorem}

\begin{proof}

 The representation is a consequence of Theorem 14 of \cite{FundBicomplex}. See also Remark 38 of \cite{BCBergman}.

The inclusion of $\phi$ in the specified class is a consequence of Theorem \ref{bctbehavior}.
\end{proof}

Using the representation from the last theorem, we prove the following generalization of statement (2) of Theorem \ref{VekHardyCombined} for functions in $H^p_{A,0}(D,\bc)$.

\begin{theorem}\label{thm: bcvekhardybzerorep}
Let $0 < p < \infty$ and $A \in L^q(D, \bc)$, $q>2$. For every $f \in H_{A,0}(D,\bc)$, $f = \varphi e^\phi\in H^{p}_{A,0}(D,\bc)$ if and only if $\varphi \in H^p(D,\bc)$. 
\end{theorem}

\begin{proof}
Suppose that $f  = \varphi e^\phi \in H_{A,0}(D,\bc)$ and $\varphi \in H^p(D,\bc)$. Then, for each $r \in (0,1)$, we have
\begin{align*}
    \int_0^{2\pi} ||f(re^{i\theta})||_{\bc}^p \,d\theta
    &=\int_0^{2\pi} ||\varphi(re^{i\theta})e^{\phi(re^{i\theta})}||_{\bc}^p \,d\theta\\
    &\leq 2^{p/2}\int_0^{2\pi} ||\varphi(re^{i\theta})||^p_\bc \,||e^{\phi(re^{i\theta})}||_{\bc}^p \,d\theta.
\end{align*}
Since $e^\phi \in C^{0,\alpha}(\overline{D},\bc)$, it follows that there exists $M>0$ such that $||e^{\phi(re^{i\theta})}||_{\bc} < M$, for all $re^{i\theta} \in \overline{D}$. Hence, 
\begin{align*}
2^{p/2}\int_0^{2\pi} ||\varphi(re^{i\theta})||^p_\bc \,||e^{\phi(re^{i\theta})}||_{\bc}^p \,d\theta
&\leq 2^{p/2}M^p\int_0^{2\pi} ||\varphi(re^{i\theta})||^p_\bc  \,d\theta\\
&\leq 2^{p/2}M^p ||\varphi||^p_{H^p_{\bc}} \,d\theta < \infty.
\end{align*}
Thus, 
\[
\sup_{0< r < 1} \int_0^{2\pi} ||f(re^{i\theta})||_{\bc}^p \,d\theta \leq 2^{p/2}M^p ||\varphi||^p_{H^p_{\bc}} \,d\theta < \infty,
\]
and $f \in H^p_{A,0}(D,\bc)$. 

Now, suppose that $f = \varphi e^\phi \in H^p_{A,0}(D,\bc)$. Since $\phi \in C^{0,\alpha}(\overline{D},\bc)$, by Theorem \ref{thm: bcvekbzerorep}, it follows that $(e^\phi)^+, (e^\phi)^- \in C^{0,\alpha}(\overline{D})$, where $e^\phi = p^+ (e^\phi)^+ + p^- (e^\phi)^-$. Since continuous functions on compact sets are bounded, it follows that there exist $M_+,M_-, m_+, m_- > 0$ such that 
\[
    m_+ \leq |(e^{\phi(re^{i\theta})})^+|\leq M_+
\]
and 
\[
    m_- \leq |(e^{\phi(re^{i\theta})})^-|\leq M_-,
\]
for all $re^{i\theta} \in \overline{D}$. Since $e^v$ is not a zero divisor for all $v \in \bc$,  it follows that $e^v$ is invertible, for all $v \in\bc$, see \cite{FundBicomplex}. Let $m:= \min\{m_+,m_-\}$ and $M:= \max\{M_+,M_-\}$. Then,
\[
m \leq |(e^{\phi(re^{i\theta})})^+|\leq M
\]
and
\[
m \leq |(e^{\phi(re^{i\theta})})^-|\leq M,
\]
for all $re^{i\theta} \in \overline{D}$. Hence, 
\[
\frac{1}{m} \geq \frac{1}{|(e^{\phi(re^{i\theta})})^+|}\geq \frac{1}{M}
\]
and
\[
\frac{1}{m} \geq \frac{1}{|(e^{\phi(re^{i\theta})})^-|}\geq \frac{1}{M},
\]
for all $re^{i\theta} \in \overline{D}$. Thus, for $r \in (0,1)$, we have
\begin{align*}
    \int_0^{2\pi} ||\varphi(re^{i\theta})||_\bc^p \,d\theta
    &= \int_0^{2\pi} ||\varphi(re^{i\theta})e^{\phi(re^{i\theta})} (e^{\phi(re^{i\theta})})^{-1}||_\bc^p \,d\theta \\
    &\leq 2^{p/2}\int_0^{2\pi} ||\varphi(re^{i\theta})e^{\phi(re^{i\theta})} ||_\bc^p \,||(e^{\phi(re^{i\theta})})^{-1}||_\bc^p \,d\theta \\
    &\leq \int_0^{2\pi} ||\varphi(re^{i\theta})e^{\phi(re^{i\theta})} ||_\bc^p \left( \frac{1}{|(e^{\phi(re^{i\theta})})^+|} + \frac{1}{|(e^{\phi(re^{i\theta})})^-|}\right)^p \,d\theta \\
     &\leq \frac{2^p}{m^p} \int_0^{2\pi} ||\varphi(re^{i\theta})e^{\phi(re^{i\theta})} ||_\bc^p  \,d\theta \\
     &= \frac{2^p}{m^p} \int_0^{2\pi} ||f(re^{i\theta}) ||_\bc^p  \,d\theta \\
     &\leq \frac{2^p}{m^p} || f||_{H^p_\bc}^p  < \infty.
\end{align*}
Therefore, 
\[
    ||\varphi||_{H^p_\bc}^p \leq C ||f||_{H^p_\bc}^p < \infty,
\]
where $C$ is a constant that depends only on $p$ and $A$, and $\varphi \in H^p(D,\bc)$.

\end{proof}

Using the above representation, we show that inclusion in certain Hardy spaces guarantees inclusion in a wide range of Lebesgue spaces. This extends statement (4) of Theorem \ref{VekHardyCombined} for the complex-valued Vekua-Hardy spaces and Proposition \ref{Prop: bchardymlessthan2p} for bicomplex-holomorphic Hardy spaces. 

\begin{theorem}\label{inBCVekHardypinLebm}
Let $p$ be a positive real number and $A\in L^q(D)$, $q>2$. Every function $w \in H^p_{A,0}(D,\bc)$ is an element of $L^m(D,\bc)$, for every $0 < m < 2p$. 
\end{theorem}

\begin{proof}
By Theorem \ref{thm: bcvekhardybzerorep}, if $w \in H^p_{A,0}(D,\bc)$, then $w = e^\phi \varphi$ and $\varphi \in H^p(D,\bc)$. Since $\varphi \in H^p(D,\bc)$, it follows, by Proposition \ref{Prop: bchardymlessthan2p}, that $\varphi \in L^m(D,\bc)$, for every $0 < m < 2p$. Since $A \in L^q(D)$, $q>2$, it follows that $e^\phi \in C^{0, \alpha}(\overline{D},\bc)$. So, there exists a positive real number $M$ such that 
\[
||e^{\phi(z)}||_{\bc} < M
\]
for all $z \in \overline{D}$. Hence, 
\begin{align*}
 \iint_D ||w(z)||_{\bc}^m \,dx\,dy
 &= \iint_D ||e^{\phi(z)} \varphi(z)||_{\bc}^m \,dx\,dy\\
 &\leq 2^{m/2}\iint_D ||e^{\phi(z)}||_{\bc}^m \,|| \varphi(z)||_{\bc}^m \,dx\,dy\\
 &\leq 2^{m/2}M^m\iint_D  || \varphi(z)||_{\bc}^m \,dx\,dy< \infty,
\end{align*}
for every $0 < m < 2p.$ Therefore, $w \in L^m(D,\bc)$, for every $0 < m < 2p$. 
\end{proof}

We recover an extension of statement (4) of Theorem \ref{VekHardyCombined} in the $B \not\equiv 0$ case so long as the hypotheses of Theorem \ref{bcvekhardyrepgeneral} are satisfied.

\begin{theorem}\label{genbcvekhardyinlebm}
Let $0 < p < \infty$. For every $f \in \hab$ such that $\dbar f= Af + B\overline{f} \in L^q(D, \bc)$, $q>2$, with representation $f = \varphi + T_\bc(Af + B\overline{f})$, $f$ is in $L^m(D,\bc)$ for every $0 < m < 2p$. The result holds when $1 < q \leq 2$ so long as $p$ satisfies $p < \frac{q}{2-q}$.
\end{theorem}

\begin{proof}
By Theorem \ref{bcvekhardyrepgeneral}, every $f \in H^p_{A,B}(D,\bc)$ such that $\dbar f= Af + B\overline{f} \in L^q(D, \bc)$, $q>2$, is representable as $f = \varphi + T_\bc(Af + B\overline{f})$ with $\varphi \in H^p(D,\bc)$. By Proposition \ref{Prop: bchardymlessthan2p}, $\varphi \in L^m(D,\bc)$, for every $0 < m < 2p$. If $q>2$, then $T_\bc(Af + B\overline{f}) \in C^{0, \alpha}(\overline{D},\bc) \subset L^m(D,\bc)$, for every $0 < m < 2p$. If $1 < q \leq 2$, then $T_\bc(Af + B\overline{f}) \in L^\gamma(D, \bc)$, for $1 < \gamma < \frac{2q}{2-q}$, by Theorem \ref{bctbehavior}. Since in this case we assume that $p < \frac{q}{2-q}$, it follows that every $m < 2p < \frac{2q}{q-2}$ and $T_\bc(Af+ B\overline{f}) \in L^m(D,\bc)$, for every $0 < m < 2p$. Thus, in either case, we have
\begin{align*}
    \iint_D ||f(z)||^m_{\bc} \,dx\,dy 
    &= C_m\left(\iint_D ||\varphi(z)||^m_{\bc} \,dx\,dy  + \iint_D ||T_\bc(Af + B\overline{f})(z)||^m_{\bc} \,dx\,dy \right) < \infty,
\end{align*}
where $C_m$ is a constant that depends only on $m$. Therefore, $f \in L^m(D)$, for every $0 < m < 2p$. 
\end{proof}

Finally, we take advantage of the structure of bicomplex numbers, specifically their idempotent representation, to show that the component functions of the idempotent representation of a solution to the $\bc$-Vekua equation themselves are solutions to certain complex Vekua-type equations. 

\begin{theorem}\label{bcvekimpliescvek}
Let $A, B \in L^q(D,\bc)$, $q>2$. A function $w:D\to\bc$ solves
\[
    \dbar w = Aw + B\overline{w}
\]
if and only if
\[
    \frac{\p (w^+)^*}{\p z^*} = (A^+)^* (w^+)^* + (B^+)^* w^+
\]
and
\[
    \frac{\p w^-}{\p z^*} = A^- w^- + B^- (w^-)^*.
\]
\end{theorem}

\begin{proof}
By Definition \ref{bcdbardef},  
\[
    \dbar = p^+ \frac{\p}{\p z} + p^- \frac{\p}{\p z^*}.
\]
By Proposition \ref{everybchasplusandminus},
\[
    w = p^+ w^+ + p^-w^-,\quad A = p^+ A^+ + p^- A^-, \text{ and } B= p^+ B^+ + p^- B^-,
\]
where $w^\pm, A^\pm, B^\pm : D\to\mathbb{C}$. Observe that 
\begin{align*}
    \dbar w &:= \left( p^+ \frac{\p}{\p z} + p^- \frac{\p}{\p z^*}\right) (p^+w^+ + p^- w^-) = p^+\frac{\p w^+}{\p z} + p^- \frac{\p w^-}{\p z^*}.
\end{align*}
So, if $w$ solves $\dbar w= Aw + B\overline{w}$, then 
\begin{align*}
    &p^+\frac{\p w^+}{\p z} + p^- \frac{\p w^-}{\p z^*} \\
    &= Aw + B\overline{w} \\
    &= (p^+ A^+ + p^- A^-)(p^+ w^+ + p^-w^-) + (p^+ B^+ + p^- B^-)\overline{(p^+ w^+ + p^-w^-)}\\
    &= (p^+ A^+ + p^- A^-)(p^+ w^+ + p^-w^-) + (p^+ B^+ + p^- B^-)(p^+ (w^+)^* + p^-(w^-)^*)\\
    &= p^+(A^+ w^+ + B^+ (w^+)^*) + p^-(A^- w^- + B^- (w^-)^*).
\end{align*}
Hence, by identifying the components of the idempotent representations on the left and right hand sides, we have
\[
    \frac{\p w^+}{\p z} = A^+ w^+ + B^+ (w^+)^*
\]
and
\[
    \frac{\p w^-}{\p z^*} = A^- w^- + B^- (w^-)^*.
\]
Since $\frac{\p w^+}{\p z} = A^+ w^+ + B^+ (w^+)^*$, it follows that 
\[
    \frac{\p (w^+)^*}{\p z^*} = (A^+)^* (w^+)^* + (B^+)^* w^+.
\]
The other direction is clear by reversing the above computation. 
\end{proof}

Using the last theorem, since the components in the idempotent representation of solutions to a $\bc$-Vekua equation are, up to complex conjugation, solutions of the classical complex Vekua equation, we can make the following connection between the $\bc$-Vekua-Hardy spaces and the $\mathbb{C}$-Vekua-Hardy spaces. This is directly inspired by the corresponding result for $\bc$-holomorphic functions which is Theorem \ref{thm: bchardyrep}.

\begin{theorem}\label{thm: bcvekhardyimpliescvekhardy}
Let $A, B \in L^q(D, \bc)$, $q>2$. For $0 < p < \infty$, $w \in H^p_{A,B}(D,\bc)$ if and only if $(w^+)^* \in H^p_{(A^+)^*,(B^+)^*}(D)$ and $w^- \in H^p_{A^-,B^-}(D)$. 
\end{theorem}

\begin{proof}
By Theorem \ref{bcvekimpliescvek}, $w \in H_{A,B}(D,\bc)$ if and only if $(w^+)^* \in H_{(A^+)^*,(B^+)^*}(D)$ and $w^- \in H_{A^-,B^-}(D)$. 

Suppose that $w \in H^p_{A,B}(D,\bc)$. Since 
\[
    |w^\pm(z)| \leq \sqrt{2}||w(z)||_{\bc},
\]
for all $z \in D$, by \eqref{bcbasicestimates}, it follows that 
\[
    \sup_{0< r < 1} \int_0^{2\pi} |w^\pm(re^{i\theta})|^p  \,d\theta \leq \sqrt{2} \sup_{0< r < 1} \int_0^{2\pi} ||w(re^{i\theta})||_{\bc}^p  \,d\theta < \infty.
\]
Since $|w^+(z)| = |(w^+)^*(z)|$, for all $z$, it follows that $(w^+)^* \in H^p_{(A^+)^*,(B^+)^*}(D)$ and $w^- \in H^p_{A^-,B^-}(D)$. 

Now, suppose that $(w^+)^* \in H^p_{(A^+)^*,(B^+)^*}(D)$ and $w^- \in H^p_{A^-,B^-}(D)$. Since 
\[
    ||w(z)||_{\bc} \leq 2^{-1/2} ( |w^+(z) + w^-(z)|),
\]
for every $z \in D$, by \eqref{bcbasicestimates}, it follows that, for each $re^{i\theta} \in D$, we have
\begin{align*}
    \int_0^{2\pi} ||w(re^{i\theta})||_{\bc}^p  \,d\theta
    &\leq 2^{-p/2}\int_0^{2\pi} |w^+(re^{i\theta})+w^-(re^{i\theta})|^p  \,d\theta \\
    &\leq C_p \left( \int_0^{2\pi} |w^+(re^{i\theta})|^p \,d\theta + \int_0^{2\pi}|w^-(re^{i\theta})|^p  \,d\theta \right)\\
    &\leq C_p \left( \
    \sup_{0< r < 1}\int_0^{2\pi} |w^+(re^{i\theta})|^p \,d\theta + \sup_{0 < r < 1}\int_0^{2\pi}|w^-(re^{i\theta})|^p  \,d\theta \right),
\end{align*}
where $C_p$ is a constant that depends only on $p$. Since the right hand side of the above inequality has no dependce on $r$, it follows that 
\[
\sup_{0 < r < 1} \int_0^{2\pi} ||w(re^{i\theta})||_{\bc}^p  \,d\theta \leq C_p \left( \
    \sup_{0< r < 1}\int_0^{2\pi} |w^+(re^{i\theta})|^p \,d\theta + \sup_{0 < r < 1}\int_0^{2\pi}|w^-(re^{i\theta})|^p  \,d\theta \right) < \infty.
\]
Therefore, $w \in H^p_{A,B}(D,\bc)$.

\end{proof}

\begin{comm}
    We end this subsection by recognizing that Casta\~neda and Kravchenko proved, in Theorem 3 of \cite{CastaKrav}, that solutions of the $\bc$-Vekua equation
    \[
        \dbar w = Aw + B\overline{w},
    \]
    with $B$ neither zero nor a zero divisor, are also representable by a similarity principle representation analogous to solutions of the complex Vekua equation. With this representation, the results for the complex-valued Vekua-Hardy spaces of \cite{KlimBook,PozHardy, CompOp, moreVekHardy, conjbel} should be immediately recoverable. We leave the details to the interested reader.  
\end{comm}

\subsection{Boundary Behavior}

In this subsection, we prove functions in certain $\bc$-Vekua-Hardy spaces have $L^p$ nontangential boundary values and the functions converge to those boundary values in the $L^p$ norm. This shows that the classical boundary behavior of the $H^p(D)$ spaces is recovered by these nonholomorphic functions.

We begin with a theorem for the case where a representation of the first kind is available and use this representation to show these functions have $L^p$ nontangential boundary values and the functions converge to these nontangential boundary values in the associated $L^p$ norm.

\begin{theorem}\label{thm: bcvekbdinsimprincase}
    For $0 < p < \infty$ and $A \in L^q(D, \bc)$, $q>2$, every $f = \varphi e^\phi \in H^p_{A,0}(D,\bc)$ such that $\varphi^+ \in H^p_w(D)$, where $\varphi = p^+ \varphi^+ + p^- \varphi^-$ and $w \in L^\gamma(D)$, and $\gamma>2$ or $1< \gamma \leq 2$ and $p$ satisfies $p < \frac{\gamma}{2- \gamma}$, has a nontangential limit $f_{nt} \in L^p(D, \bc)$, and 
    \[
        \lim_{r \nearrow 1} \int_0^{2\pi} ||f_{nt}(e^{i\theta}) - f(re^{i\theta}) ||^p_{\bc} \, d\theta  = 0. 
    \]  
\end{theorem}

\begin{proof}
By Theorems \ref{thm: bcvekbzerorep} and \ref{thm: bcvekhardybzerorep}, if $A \in L^q(D,\bc)$, $q>2$ and $f \in H^p_{A,0}(D)$, then $f = \varphi e^\phi$ and $\varphi \in H^p(D,\bc)$. By Theorem \ref{bcholobvcon}, since we assume that $\varphi^+ \in H^p_w(D)$ and $w$ and $p$ satisfy the relationship that $w \in L^\gamma(D)$ with $\gamma>2$ or $p < \frac{\gamma}{2-\gamma}$, it follows that $\varphi$ has a nontangential boundary value $\varphi_{nt} \in L^p(\p D,\bc)$ and 
\[
        \lim_{r \nearrow 1} \int_0^{2\pi} ||\varphi_{nt}(e^{i\theta}) - \varphi(re^{i\theta}) ||^p_{\bc} \, d\theta  = 0. 
    \]
Since $f = \varphi e^\phi$, $\varphi_{nt}$ exists and is in $L^p(\p D,\bc)$, and $e^\phi \in C^{0, \alpha}(\overline{D},\bc)$ by Theorem \ref{thm: bcvekbzerorep}, it follows that $f_{nt}$ is in $L^p(\p D,\bc)$. Also, for $r \in (0,1)$, we have 
\begin{align*}
    &\int_0^{2\pi} ||f_{nt}(e^{i\theta}) - f(re^{i\theta}) ||^p_{\bc} \, d\theta \\
    &= \int_0^{2\pi} ||\varphi_{nt}(e^{i\theta})e^{\phi(e^{i\theta})} - \varphi(re^{i\theta})e^{\phi(re^{i\theta})} ||^p_{\bc} \, d\theta \\
    &= \int_0^{2\pi} ||\varphi_{nt}(e^{i\theta})e^{\phi(e^{i\theta})} -\varphi(re^{i\theta})e^{\phi(e^{i\theta})} + \varphi(re^{i\theta})e^{\phi(e^{i\theta})} - \varphi(re^{i\theta})e^{\phi(re^{i\theta})} ||^p_{\bc} \, d\theta \\
    &= \int_0^{2\pi} ||(\varphi_{nt}(e^{i\theta}) -\varphi(re^{i\theta}))e^{\phi(e^{i\theta})} + \varphi(re^{i\theta})(e^{\phi(e^{i\theta})} - e^{\phi(re^{i\theta})}) ||^p_{\bc} \, d\theta \\
    &\leq  C_p \left( \int_0^{2\pi} ||(\varphi_{nt}(e^{i\theta}) -\varphi(re^{i\theta}))e^{\phi(e^{i\theta})}||_{\bc}^p \,d\theta + \int_0^{2\pi} ||\varphi(re^{i\theta})(e^{\phi(e^{i\theta})} - e^{\phi(re^{i\theta})}) ||^p_{\bc} \, d\theta \right) \\
    &\leq  C_p \left( 2^{p/2}\int_0^{2\pi} ||(\varphi_{nt}(e^{i\theta}) -\varphi(re^{i\theta}))||_{\bc}^p \,||e^{\phi(e^{i\theta})}||_{\bc}^p \,d\theta + \int_0^{2\pi} ||\varphi(re^{i\theta})(e^{\phi(e^{i\theta})} - e^{\phi(re^{i\theta})}) ||^p_{\bc} \, d\theta \right) \\
     &\leq  C_p \left( 2^{p/2}M \int_0^{2\pi} ||(\varphi_{nt}(e^{i\theta}) -\varphi(re^{i\theta}))||_{\bc}^p  \,d\theta + \int_0^{2\pi} ||\varphi(re^{i\theta})(e^{\phi(e^{i\theta})} - e^{\phi(re^{i\theta})}) ||^p_{\bc} \, d\theta \right),
\end{align*}
where $C_p$ is a constant that depends on only $p$ and $M$ is a constant that depends on only $p$ and $A$. Since
\begin{align*}
    &\varphi(re^{i\theta})(e^{\phi(e^{i\theta})} - e^{\phi(re^{i\theta})})\\
    &= p^+[\varphi^+(re^{i\theta}) ((e^{\phi(e^{i\theta})})^+-(e^{\phi(re^{i\theta})})^+)] + p^-[\varphi^-(re^{i\theta}) ((e^{\phi(e^{i\theta})})^- -(e^{\phi(re^{i\theta})})^-)],
\end{align*}
it follows that 
\begin{align*}
&\int_0^{2\pi} ||\varphi(re^{i\theta})(e^{\phi(e^{i\theta})} - e^{\phi(re^{i\theta})}) ||^p_{\bc} \, d\theta\\
&=\int_0^{2\pi} ||p^+[\varphi^+(re^{i\theta}) ((e^{\phi(e^{i\theta})})^+-(e^{\phi(re^{i\theta})})^+)] + p^-[\varphi^-(re^{i\theta}) ((e^{\phi(e^{i\theta})})^- -(e^{\phi(re^{i\theta})})^-)] ||^p_{\bc} \, d\theta \\
&\leq 2^{-p/2}  \int_0^{2\pi} |\varphi^+(re^{i\theta}) ((e^{\phi(e^{i\theta})})^+-(e^{\phi(re^{i\theta})})^+)|^p \,d\theta \\
&\quad\quad + 2^{-p/2} \int_0^{2\pi} |\varphi^-(re^{i\theta}) ((e^{\phi(e^{i\theta})})^- -(e^{\phi(re^{i\theta})})^-) |^p \, d\theta .
\end{align*}
So, 
\begin{align*}
&C_p \,\left( 2^{p/2}M \int_0^{2\pi} ||(\varphi_{nt}(e^{i\theta}) -\varphi(re^{i\theta}))||_{\bc}^p  \,d\theta + \int_0^{2\pi} ||\varphi(re^{i\theta})(e^{\phi(e^{i\theta})} - e^{\phi(re^{i\theta})}) ||^p_{\bc} \, d\theta \right)\\
&\leq C_p \, 2^{p/2}M \int_0^{2\pi} ||(\varphi_{nt}(e^{i\theta}) -\varphi(re^{i\theta}))||_{\bc}^p  \,d\theta \\
&\quad\quad +  C_p\,  2^{-p/2}  \int_0^{2\pi} |\varphi^+(re^{i\theta}) ((e^{\phi(e^{i\theta})})^+-(e^{\phi(re^{i\theta})})^+)|^p \,d\theta \\
&\quad\quad + C_p \,2^{-p/2} \int_0^{2\pi} |\varphi^-(re^{i\theta}) ((e^{\phi(e^{i\theta})})^- -(e^{\phi(re^{i\theta})})^-) |^p \, d\theta
\end{align*}
Observe that the first integral summand in the right hand side of the above inequality will converge to $0$ as $r \nearrow 1$ by Theorem \ref{bcholobvcon}, and the second and third integral summands in the right hand side of the above inequality will converge to $0$ by Lebesgue's Dominated Convergence Theorem as $r \nearrow 1$. Thus, 
\[
\lim_{r \nearrow 1} \int_0^{2\pi} ||f_{nt}(e^{i\theta}) - f(re^{i\theta}) ||^p_{\bc} \, d\theta = 0.
\]

\end{proof}

Next, we extend statement (3) of Theorem \ref{VekHardyCombined} and Theorem \ref{bcbdinsfuncins} to this setting. 

\begin{theorem}\label{bcvekhardybdinsfuncins}
Let $p$ be a positive real number and $A \in L^q(D)$, $q>2$. Every $f \in H^p_{A,0}(D,\bc)$ such that $f_{nt} \in L^s(\p D, \bc)$, where $s > p$, is an element of $H^s_{A,0}(D,\bc)$.
\end{theorem}

\begin{proof}
By Theorem \ref{thm: bcvekhardybzerorep}, $f = e^\phi \varphi$, and $\varphi \in H^p(D,\bc)$. Since $A \in  L^q(D)$, $q>2$, it follows that $e^\phi \in C^{0, \alpha}(\overline{D},\bc)$. So, there exist real numbers $m>0$ and $M>0$ such that 
\[
m < ||e^{\phi(z)}||_{\bc} < M
\]
for all $z \in \overline{D}$. So, 
\begin{align*}
\int_0^{2\pi} ||\varphi_{nt}(e^{i\theta})||_{\bc}^s \,d\theta
&= \int_0^{2\pi} \left|\left|\frac{e^{\phi(e^{i\theta})}}{e^{\phi(e^{i\theta})}}\varphi_{nt}(e^{i\theta})\right|\right|_{\bc}^s \,d\theta\\
&\leq \frac{2^{s/2}}{m^s} \int_0^{2\pi} ||e^{\phi(e^{i\theta})}\varphi_{nt}(e^{i\theta})||^s_{\bc} \,d\theta\\
&= \frac{2^{s/2}}{m^s} \int_0^{2\pi} ||f_{nt}(e^{i\theta})||^s_{\bc} \,d\theta < \infty
\end{align*}
and $\varphi_{nt} \in L^s(\p D, \bc)$. By Theorem \ref{bcbdinsfuncins}, $\varphi \in H^s(D, \bc)$. By Theorem \ref{thm: bcvekhardybzerorep}, $f \in H^s_{A,0}(D,\bc)$. 
\end{proof}

Now, using the representation of the second kind from the previous subsection, we show the Hardy space boundary behavior extends to the $B \not\equiv 0$ case also. 

\begin{theorem}\label{Thm: bcvekhardybvcon}
Let $0 < p < \infty$. Every $f \in \hpab$ such that $\dbar f= Af + B\overline{f} \in L^q(D, \bc)$, $q>2$, with representation $f = \varphi + T_\bc(Af + B\overline{f})$ and $\varphi^+ \in H^p_w(D)$, where $w \in L^\gamma$, $\gamma>2$ or $1 < \gamma\leq 2$ and $p < \frac{\gamma}{2-\gamma}$, has a nontangential boundary value $f_{nt} \in L^p(\p D)$ and 
\[
\lim_{r \nearrow 1} \int_0^{2\pi} ||f_{nt}(e^{i\theta}) - f(re^{i\theta}) ||^p_{\bc} \, d\theta = 0
\]
The result holds when $1 < q \leq 2$ so long as $p$ satisfies $p < \frac{q}{2-q}$. 
\end{theorem}

\begin{proof}
By Theorem \ref{bcvekhardyrepgeneral}, if $f \in H^p_{A,B}(D,\bc)$ with $\dbar f = Af + B\overline{f} \in L^q(D)$, where $q$ and $p$ satisfy the hypothesis, then $f = \varphi + T_\bc(Af + B\overline{f})$ and $\varphi \in H^p(D,\bc)$. By Theorem \ref{bcholobvcon}, $\varphi$ has a nontangential boundary value $\varphi_{nt} \in L^p(\p D,\bc)$ and 
\[
        \lim_{r \nearrow 1} \int_0^{2\pi} ||\varphi_{nt}(e^{i\theta}) - \varphi(re^{i\theta}) ||^p_{\bc} \, d\theta  = 0. 
    \]
By Theorem \ref{bctbehavior}, if $p$ and $q$ satisfy the hypothesis, then $T_\bc(Af + B\overline{f}) \in L^p(\p D,\bc)$ and 
\[
\lim_{r \nearrow 1} \int_0^{2\pi} ||T_\bc(Af + B\overline{f})(e^{i\theta}) - T_\bc(Af + B\overline{f})(re^{i\theta}) ||^p_{\bc} \, d\theta  = 0.
\]
So, $f_{nt}$ exists and is in $L^p(\p D,\bc)$. Also, for $r \in (0,1)$, we have
\begin{align*}
    &\int_0^{2\pi} ||f_{nt}(e^{i\theta}) - f(re^{i\theta}) ||^p_{\bc} \, d\theta \\
    &= \int_0^{2\pi} ||\varphi_{nt}(e^{i\theta}) + T_\bc(Af + B\overline{f})(e^{i\theta}) - (\varphi(re^{i\theta}) + T_\bc(Af + B\overline{f})(re^{i\theta})) ||^p_{\bc} \, d\theta \\
    &\leq C_p \left( \int_0^{2\pi} ||\varphi_{nt}(e^{i\theta}) - \varphi(re^{i\theta})||_\bc^p\,d\theta  + \int_0^{2\pi} || T_\bc(Af + B\overline{f})(e^{i\theta}) - T_\bc(Af + B\overline{f})(re^{i\theta}) ||^p_{\bc} \, d\theta\right) .
\end{align*}
Therefore, 
\begin{align*}
& \lim_{r \nearrow 1} \int_0^{2\pi} ||f_{nt}(e^{i\theta}) - f(re^{i\theta}) ||^p_{\bc} \, d\theta \\
&\leq \lim_{r \nearrow 1} C_p \left( \int_0^{2\pi} ||\varphi_{nt}(e^{i\theta}) - \varphi(re^{i\theta})||_\bc^p\,d\theta  + \int_0^{2\pi} || T_\bc(Af + B\overline{f})(e^{i\theta}) - T_\bc(Af + B\overline{f})(re^{i\theta}) ||^p_{\bc} \, d\theta\right)\\
&\leq  C_p \left( \lim_{r \nearrow 1}\int_0^{2\pi} ||\varphi_{nt}(e^{i\theta}) - \varphi(re^{i\theta})||_\bc^p\,d\theta  + \lim_{r \nearrow 1}\int_0^{2\pi} || T_\bc(Af + B\overline{f})(e^{i\theta}) - T_\bc(Af + B\overline{f})(re^{i\theta}) ||^p_{\bc} \, d\theta\right)\\
&= 0.
\end{align*}

\end{proof}

In \cite{BCAtomic}, $\bc$-holomorphic Hardy space functions were shown to recover the boundary behavior of Hardy-type classes by appealing to the idempotent representation for those functions. These results are Theorems \ref{thm: bchardyrep} and \ref{bcholobvcon} above, see \cite{BCAtomic} for the proofs. Using the idempotent representation from Theorem \ref{thm: bcvekhardyimpliescvekhardy} above for $\bc$-Vekua-Hardy space functions, we appeal to the same argument as in the $\bc$-holomorphic case to show that Hardy-type boundary behavior extends in another case for $\bc$-Vekua-Hardy spaces. Specifically, we need not assume that $B \equiv 0$ as in Theorem \ref{thm: bcvekbdinsimprincase} or that the entire right hand side of the equation is in $L^q(D,\bc)$, $q>2$, as in Theorem \ref{Thm: bcvekhardybvcon}. We need only that the $p^+$ component function from the idempotent representation be an element of a certain generalized Hardy class, as defined in Definition \ref{complexnonhomoghardydeff}. 

\begin{theorem}\label{bcvekhardybvfromidempotent}
 For $A,B \in L^q(D,\bc)$, $q>2$, and $0 < p < \infty$, every $w = p^+ w^+ + p^- w^-\in H^p_{A,B}(D,\bc)$ such that $w^+ \in H^p_g(D)$, where $g \in L^\ell(D)$, and $\ell >2$ or $1 < \ell \leq 2$ and $p$ satisfies $p < \frac{\ell}{2-\ell}$, has a nontangential limit $w_{nt} \in L^p(\p D, \bc)$, and 
    \[
        \lim_{r \nearrow 1} \int_0^{2\pi} ||w_{nt}(e^{i\theta}) - w(re^{i\theta}) ||^p_{\bc} \, d\theta  = 0. 
    \] 
\end{theorem}

\begin{proof}
Let $w \in H^p_{A,B}(D,\bc)$. By Theorem \ref{thm: bcvekhardyimpliescvekhardy}, $w = p^+w^+ + p^- w^-$, where $(w^+)^* \in H^p_{(A^+)^*,(B^+)^*}(D,\bc)$ and $w^- \in H^p_{A^-,B^-}(D)$. Suppose also that $w^+ \in H^p_g(D)$, for some $g \in L^\ell(D)$, $\ell > 2$ (or if $1 < \ell \leq 2$, then $p$ satisfies $p < \frac{\ell}{2-\ell}$). By Theorem \ref{thm: nonhomogHpbvcon} and statement (3) of Theorem \ref{VekHardyCombined}, it follows that $w^+$ and $w^-$ have nontangential boundary values $w^+_{nt}$ and $w^-_{nt}$, respectively, $w^+_{nt}, w^-_{nt} \in L^p(\p D)$, and,
\[
\lim_{r \nearrow 1} \int_0^{2\pi} |w^+_{nt}(e^{i\theta}) - w^+(re^{i\theta}) |^p \, d\theta  = 0,
\]
and 
\[
\lim_{r \nearrow 1} \int_0^{2\pi} |w^-_{nt}(e^{i\theta}) - w^-(re^{i\theta}) |^p \, d\theta  = 0.
\]
By linearity, $w_{nt} = p^+ w^+_{nt} + p^- w^-_{nt}$ exists and is in $L^p(\p D,\bc)$. Observe that, for each $r \in (0,1)$, 
\begin{align*}
&\int_0^{2\pi} ||w_{nt}(e^{i\theta}) - w(re^{i\theta}) ||^p_{\bc} \, d\theta \\
&= \int_0^{2\pi} ||p^+ w^+_{nt}(e^{i\theta}) + p^- w^-_{nt}(e^{i\theta}) - (p^+ w^+(re^{i\theta}) + p^- w^-(re^{i\theta})) ||^p_{\bc} \, d\theta \\
&= \int_0^{2\pi} ||p^+ (w^+_{nt}(e^{i\theta})-w^+(re^{i\theta})) + p^- (w^-_{nt}(e^{i\theta}) - w^-(re^{i\theta})) ||^p_{\bc} \, d\theta \\
&\leq \int_0^{2\pi} \left(\frac{1}{\sqrt{2}}(|w^+_{nt}(e^{i\theta})-w^+(re^{i\theta})| + |w^-_{nt}(e^{i\theta}) - w^-(re^{i\theta}) |)\right)^p \, d\theta \\
&\leq \frac{C_p}{2^{p/2}}\left(\int_0^{2\pi} |w^+_{nt}(e^{i\theta})-w^+(re^{i\theta})|^p \,d\theta + \int_0^{2\pi} |w^-_{nt}(e^{i\theta}) - w^-(re^{i\theta}) |^p \, d\theta \right), 
\end{align*}
where $C_p$ is a constant that depends only on $p$. Since 
\[
\lim_{r \nearrow 1} \int_0^{2\pi} |w^+_{nt}(e^{i\theta}) - w^+(re^{i\theta}) |^p \, d\theta  = 0
\]
and 
\[
\lim_{r \nearrow 1} \int_0^{2\pi} |w^-_{nt}(e^{i\theta}) - w^-(re^{i\theta}) |^p \, d\theta  = 0,
\]
it follows that 
\[
\lim_{r \nearrow 1}\frac{C_p}{2^{p/2}}\left(\int_0^{2\pi} |w^+_{nt}(e^{i\theta})-w^+(re^{i\theta})|^p \,d\theta + \int_0^{2\pi} |w^-_{nt}(e^{i\theta}) - w^-(re^{i\theta}) |^p \, d\theta \right) = 0.
\]
Therefore, 
\[
\lim_{r \nearrow 1} \int_0^{2\pi} ||w_{nt}(e^{i\theta}) - w(re^{i\theta}) ||^p_{\bc} \, d\theta = 0.
\]

\end{proof}

\subsection{Distributional Boundary Values}

In this subsection, we show that certain existence results for boundary values in the sense of distributions of complex holomorphic (see \cite{GHJH2, Straube}), complex generalized analytic (see \cite{BerHou, WBD}), and bicomplex holomorphic (see \cite{BCAtomic}) functions extend to bicomplex generalized analytic functions.

We begin with a theorem that extends statement (5) of Theorem \ref{VekHardyCombined}, one direction of which was proved in \cite{BerHou}, to solutions of $\bc$-Vekua equations.

\begin{theorem}
    For $p \geq 1$ and $A, B \in L^q(D)$, $q>2$, $w \in H^p_{A,B}(D,\bc)$ if and only if $w \in H_{A,B}(D,\bc)$, $w_b$ exists, and $w_b \in L^p(\p D, \bc)$. 
\end{theorem}

\begin{proof}
First, suppose $w \in H_{A,B}(D,\bc)$, $w_b$ exists, and $w_b \in L^p(\p D)$. Since $w_b$ exists and is in $L^p(\p D,\bc)$, it follows, by Proposition \ref{propLqiff}, that $w_b^\pm$ exist and $w^\pm \in L^p(\p D, \mathbb{C})$. Since 
\begin{align*}
\int_0^{2\pi} (w^+(re^{i\theta}))^*\,\phi(\theta) \,d\theta = \left(\int_0^{2\pi} w^+(re^{i\theta})\,\phi(\theta) \,d\theta \right)^*,
\end{align*}
for every real-valued $\phi \in C^\infty(\p D)$ and $re^{i\theta} \in D$, and $w^+_b$ exists, it follows that $(w^+)^*_b$ exists. By Theorem \ref{thm: bcvekhardyimpliescvekhardy}, $(w^+)^* \in H_{(A^+)^*, (B^+)^*}(D)$ and $w^- \in H_{A^-,B^-}(D)$. Since $(w^+)^* \in H_{(A^+)^*, (B^+)^*}(D)$ and $w^- \in H_{A^-,B^-}(D)$, $(w^+)^*_b$ and $w^-_b$ both exist, and $(w^+)^*_b,w^-_b \in L^p(\p D)$, it follows by statement (5) of Theorem \ref{VekHardyCombined} that $(w^+)^* \in H^p_{(A^+)^*, (B^+)^*}(D)$ and $w^- \in H^p_{A^-,B^-}(D)$. By Theorem \ref{thm: bcvekhardyimpliescvekhardy}, $w \in H^p_{A,B}(D,\bc)$. 

Now, suppose $w \in H^p_{A,B}(D,\bc)$. By Theorem \ref{thm: bcvekhardyimpliescvekhardy}, $(w^+)^* \in H^p_{(A^+)^*, (B^+)^*}(D)$ and $w^- \in H^p_{A^-,B^-}(D)$. By statement (3) of Theorem \ref{VekHardyCombined}, the nontangential boundary values $(w^+)^*_{nt}$ and $w^-_{nt}$ exist and are in $L^p(\p D)$. By statement (4) of Theorem \ref{VekHardyCombined},  $(w^+)^*, w^- \in L^1(D)$. By part (3) of Theorem \ref{VekHardyCombined} again, 
\[
\lim_{r\nearrow 1}\int_0^{2\pi} |(w^+)^*(re^{i\theta}) - (w^+)^*_{nt}(e^{i\theta})|^p \,d\theta = 0
\]
and
\[
\lim_{r\nearrow 1}\int_0^{2\pi} |w^-(re^{i\theta}) - w^-_{nt}(e^{i\theta})|^p \,d\theta = 0.
\]
By Theorem \ref{lonedistbv}, since $(w^+)^*, w^- \in L^1(D)$, $(w^+)^*_{nt}, w^-_{nt} \in L^p(\p D) \subset L^1(\p D)$, and 
\[
\lim_{r\nearrow 1}\int_0^{2\pi} |(w^+)^*(re^{i\theta}) - (w^+)^*_{nt}(e^{i\theta})| \,d\theta = 0
\]
and
\[
\lim_{r\nearrow 1}\int_0^{2\pi} |w^-(re^{i\theta}) - w^-_{nt}(e^{i\theta})| \,d\theta = 0.
\]
it follows that $(w^+)^*_{b}$ and $w^-_{b}$ exist. Also, by Theorem \ref{lonedistbv} $(w^+)^*_{nt} = (w^+)^*_{b}$ and $w^-_{nt} = w^-_{b}$.

\end{proof}

Next, we show that in the $A \in L^q(D,\bc)$, $q>2$ and $ B\equiv 0$ case, i.e., where we have a representation of the first kind, there is a condition that mirrors the statement of Theorem \ref{bcGHJH23point1} for existence of distributional boundary values for solutions to the corresponding $\bc$-Vekua equation. This directly extends Lemma 5.9 from \cite{WBD} to the bicomplex setting, and the proof closely follows the argument of that result.

\begin{prop}
Let $A \in L^q(D, \bc)$, $q>2$, and $w \in H_{A,0}(D,\bc)$. The $\bc$-holomorphic factor $\varphi$ (appearing in the representation $w = \varphi e^\phi$ provided by Theorem \ref{thm: bcvekbzerorep}) has a boundary value in the sense of distributions $\varphi_b$ if and only if there exist constants $C>0$ and $\alpha \geq 0$ such that 
\[
    ||w(re^{i\theta})||_{\bc} \leq \frac{C}{(1-r)^\alpha},
\]
for all $re^{i\theta} \in D$.
\end{prop}

\begin{proof}
By Theorem \ref{thm: bcvekbzerorep}, every $w \in H_{A,0}(D,\bc)$, where $A \in L^1(D)$, has a representation 
\[
    w = \varphi e^\phi,
\]
where $\varphi \in Hol(D,\bc)$ and, if $q>2$, $\phi \in C^{0, \alpha}(\overline{D},\bc) \subset L^\infty(\overline{D},\bc)$. 

Suppose $\varphi_b$ exists. By Theorem \ref{bcGHJH23point1}, there exists $C> 0$ and $\alpha \geq 0$ such that 
\[
        ||\varphi(re^{i\theta})||_{\bc} \leq \frac{C}{(1-r)^\alpha},
\]
for all $re^{i\theta} \in D$. Observe that 
\begin{align*}
    ||w(re^{i\theta})||_{\bc} 
    &= ||\varphi(re^{i\theta}) e^{\phi(re^{i\theta})}||_{\bc} \\
    &\leq \sqrt{2}||\varphi(re^{i\theta})||_{\bc}\,||e^{\phi(re^{i\theta})}||_{\bc}\\
    &\leq \sqrt{2}||\varphi(re^{i\theta})||_{\bc}\,||e^{\phi}||_{L^\infty}\\
    &\leq \sqrt{2}\,||e^{\phi}||_{L^\infty}\,\frac{C}{(1-r)^\alpha},
\end{align*}
for every $re^{i\theta} \in D$. Letting $\widetilde{C}:= C\sqrt{2}\,||e^{\phi}||_{L^\infty}$, we have
\[
||w(re^{i\theta})||_{\bc}  \leq \frac{\widetilde{C}}{(1-r)^\alpha}
\]
for every $re^{i\theta} \in D$.

Now, assume there are constants $C>0$ and $\alpha \geq 0$ such that 
\[
||w(re^{i\theta})||_{\bc}  \leq \frac{C}{(1-r)^\alpha}
\]
for every $re^{i\theta} \in D$. Since $e^\phi \in C^{0,\alpha}(\overline{D},\bc)$ and $e^\phi$ is neither zero nor a zero divisor, it follows that $\left|\left|\frac{1}{e^{\phi(re^{i\theta})}}\right|\right|_{\bc}$ is bounded, i.e., there exists $N>0$ such that
\[
    \left|\left|\frac{1}{e^{\phi(re^{i\theta})}}\right|\right|_{\bc} < N,
\]
for all $re^{i\theta} \in D$. Therefore, 
\begin{align*}
||\varphi(re^{i\theta})||_{\bc} 
&= \left|\left|\frac{w(re^{i\theta})}{e^{\phi(re^{i\theta})}}\right|\right|_{\bc} \\
&\leq \sqrt{2}\,\left|\left|\frac{1}{e^{\phi(re^{i\theta})}}\right|\right|_{\bc}\, ||w(re^{i\theta})||_{\bc} \\
&\leq \sqrt{2}\,N\, ||w(re^{i\theta})||_{\bc}\\
&\leq \sqrt{2}\,N\, \frac{C}{(1-r)^\alpha},
\end{align*}
for every $re^{i\theta} \in D$. By Theorem \ref{bcGHJH23point1}, $\varphi_b$ exists. 
\end{proof}

\section{Bicomplex Polyanalytic Hardy Spaces}\label{BCPolyHardyspaces}

In this section, we consider solutions to a higher-order bicomplex differential equation that generalizes the polyanalytic functions studied by Balk \cite{Balk} (see also \cite{NVPolyText}). Specifically, we show that certain properties of the poly-Hardy spaces of \cite{polyhardy} extend to the bicomplex setting. 

\subsection{Definition}

We begin by defining the bicomplex polyanalytic functions and their associated Hardy spaces.

\begin{deff}
    Let $n$ be a positive integer. We define the bicomplex polyanalytic functions $P^n(D,\bc)$ to be the set of functions $f: D \to\bc$ such that 
    \[
        \dbar^n f = 0. 
    \]
\end{deff}

\begin{deff}
    Let $0 < p < \infty$ and $n$ be a positive integer. We define the bicomplex polyanalytic Hardy spaces $P^{n,p}(D,\bc)$ to be the set of functions $f \in P^n(D,\bc)$ such that 
    \[
        \sum_{k = 0}^{n-1}||\dbar^k f||_{H^p_{\bc}} < \infty. 
    \]
\end{deff}

\subsection{Representation}

It is well known that the complex polyanalytic functions are precisely the polynomials in $z^*$ with coefficients in $Hol(D)$. We show that this structure is inherited by $\bc$-polyanalytic functions where the role of $z^*$ is played by its bicomplexification $\widehat{z^*}$ and the coefficients are elements of $Hol(D,\bc)$. The structure of the proof follows that of the complex case from \cite{Balk}.

\begin{theorem}\label{bcpolyrep}
    Let $n$ be a positive integer. Every $f \in P^{n}(D,\bc)$ is representable as 
    \[
        f = \sum_{k = 0}^{n-1} \widehat{z^*}^k \varphi_k,
    \]
    where $\varphi_k \in \holb$, for each $k$. 
\end{theorem}

\begin{proof}
Observe that functions of the form $f(z) = \sum_{k = 0}^{n-1} \widehat{z^*}^k \varphi_k(z)$, where $\varphi_k \in \holb$, for every $k$, are elements of $P^n(D,\bc)$ by direct computation. 

We prove the other direction by induction. The result is clear for $n = 1$. Suppose that, for every $n$ satisfying $1 \leq n \leq m-1$, that if $f \in P^n(D, \bc)$, then there exists a collection $\{\varphi_k\}_{k=0}^{n-1}$ such that $\varphi_k\in \holb$, for every $k$, and
\[
    f = \sum_{k = 0}^{n-1} \widehat{z^*}^k \varphi_k
\]
Let $g \in P^m(D, \bc)$. So, $\dbar g \in P^{m-1}(D, \bc)$, and there exists a collection $\{\varphi_k\}_{k=0}^{m-2}$ such that $\varphi_k \in \holb$, for every $k$, and 
\[
    \dbar g = \sum_{k=0}^{m-2} \widehat{z^*}^k \varphi_k.
\]
Observe that 
\begin{align*}
    \dbar\left( g - \sum_{k=0}^{m-2} \frac{1}{k+1}\widehat{z^*} ^{k+1} \varphi_k\right)  
    &= \dbar g - \dbar \left( \sum_{k=0}^{m-2} \frac{1}{k+1}\widehat{z^*} ^{k+1} \varphi_k\right) \\
    &= \sum_{k=0}^{m-2} \widehat{z^*}^k \varphi_k - \sum_{k=0}^{m-2} \widehat{z^*}^k \varphi_k = 0.
\end{align*}
Thus, $g - \sum_{k=0}^{m-2} \frac{1}{k+1}\widehat{z^*} ^{k+1} \varphi_k \in \holb$. Let 
\[
\gamma = g - \sum_{k=0}^{m-2} \frac{1}{k+1}\widehat{z^*} ^{k+1} \varphi_k.
\]
Let $\phi_0 := \gamma$ and $\phi_{k+1} := \frac{1}{k+1} \varphi_k$, for every $0 \leq k \leq m-2$. Then we rearrange to have
\[
    g = \sum_{j = 0}^{m-1} \widehat{z^*}^j \phi_j.
\]

\end{proof}

The next theorem generalizes Theorem 2.1 of \cite{polyhardy} which shows that a function is in the complex poly-Hardy space if and only if the coefficients in its polynomial representation are in $H^p(D)$. 

\begin{theorem}\label{thm: polyhardyrep}
    Let $0 < p < \infty$ and $n$ be a positive integer. A function $f$ is in $P^{n,p}(D,\bc)$ with representation 
    \[
        f = \sum_{k = 0}^{n-1} \widehat{z^*}^k \varphi_k
    \]
    if and only if $\varphi_k  \in \hpb$, for each $k$. 
\end{theorem}

\begin{proof}
Suppose that $\{\varphi_k\}_{k=0}^{n-1}\subset H^p(D,\bc)$. Then 
\[
    f = \sum_{k=0}^{n-1}\widehat{z^*}^k\varphi_k
\]
solves
\[
    \frac{\p^n f}{\p\z^n} = 0,
\]
and, for each $r \in (0,1)$,  
\begin{align*}
    \int_0^{2\pi} ||f(re^{i\theta})||_{\bc}^p \,d\theta 
    &= \int_0^{2\pi} \left|\left|\sum_{k=0}^{n-1}\widehat{(re^{i\theta})^*}^k\varphi_k(re^{i\theta})\right|\right|_{\bc}^p \,d\theta\\
    &\leq C\sum_{k=0}^{n-1}\int_0^{2\pi} \left|\left|\widehat{(re^{i\theta})^*}^k\varphi_k(re^{i\theta})\right|\right|_{\bc}^p \,d\theta\\
    &\leq C\sum_{k=0}^{n-1}2^{p/2}\int_0^{2\pi} \left|\left|\widehat{(re^{i\theta})^*}^k\right|\right|_\bc^p\left|\left|\varphi_k(re^{i\theta})\right|\right|_{\bc}^p \,d\theta\\
    &\leq C \sum_{k=0}^{n-1}2^{p/2}\int_0^{2\pi} \left|\left|\varphi_k(re^{i\theta})\right|\right|_{\bc}^p \,d\theta < \infty,
\end{align*}
where $C$ is a constant that depends on only $p$ and $n$. Thus, 
\[
 \sup_{0 < r < 1}\int_0^{2\pi} ||f(re^{i\theta})||_{\bc}^p \,d\theta \leq  C \sum_{k=0}^{n-1}2^{p/2}\int_0^{2\pi} \left|\left|\varphi_k(re^{i\theta})\right|\right|_{\bc}^p \,d\theta < \infty,
\]
and $f \in P^{n,p}(D, \bc)$. 

Now, suppose that $f \in P^{n,p}(D,\bc)$. By Theorem \ref{bcpolyrep}, it follows that there exist $\varphi_k \in \holb$, for every $k$, such that 
\[
    f = \sum_{k=0}^{n-1}\widehat{z^*}^k\varphi_k.
\]  
Observe that the explicit form of $\varphi_k$ can be given by 
\begin{align*}
    \varphi_k(z) = \dfrac{1}{k!}\sum_{j=0}^{n-1-k}\dfrac{(-1)^j}{j!}\widehat{z^*}^j\, \dbar^{k+j}f, 
\end{align*}
for each $k$ (see \cite{l2poly} or \cite{itvek, itvekbvp, WBD} for the motivation for this formula in the complex setting). Hence, for each $k$ and $r \in (0,1)$, we have
\begin{align*}
\int_0^{2\pi} \left|\left|\varphi_k(re^{i\theta})\right|\right|_{\bc}^p \,d\theta
&= \int_0^{2\pi} \left|\left|\dfrac{1}{k!}\sum_{j=0}^{n-1-k}\dfrac{(-1)^j}{j!}\widehat{(re^{i\theta})^*}^j\, \dbar^{k+j}f(re^{i\theta})\right|\right|_{\bc}^p \,d\theta \\
&\leq C_p \sum_{j=0}^{n-1-k}\int_0^{2\pi} \left|\left|\dbar^{k+j}f(re^{i\theta})\right|\right|_{\bc}^p \,d\theta\\
&\leq C_p \sum_{j=0}^{n-1-k}\sup_{0 < r < 1}\int_0^{2\pi} \left|\left|\dbar^{k+j}f(re^{i\theta})\right|\right|_{\bc}^p \,d\theta,
\end{align*}
where $C_p$ is a constant that depends on only $p$ and $n$. Since $f \in P^{n,p}(D, \bc)$, it follows that 
\[
\sup_{0 < r < 1}\int_0^{2\pi} \left|\left|\dbar^{\ell}f(re^{i\theta})\right|\right|_{\bc}^p \,d\theta < \infty,
\]
for all $\ell$ satisfying $0 \leq \ell \leq n-1$. Thus, 
\[
C_p \sum_{j=0}^{n-1-k}\sup_{0 < r < 1}\int_0^{2\pi} \left|\left|\dbar^{k+j}f(re^{i\theta})\right|\right|_{\bc}^p \,d\theta < \infty,
\]
so
\[
\sup_{0 < r < 1} \int_0^{2\pi}\left|\left|\varphi_k(re^{i\theta})\right|\right|_{\bc}^p \,d\theta \leq C_p \sum_{j=0}^{n-1-k}\sup_{0 < r < 1}\int_0^{2\pi} \left|\left|\dbar^{k+j}f(re^{i\theta})\right|\right|_{\bc}^p \,d\theta < \infty.
\]
Therefore, $\varphi_k \in H^p(D,\bc)$, for every $k$. 

\end{proof}

Using the representation from the last theorem, we now show that Proposition \ref{Prop: bchardymlessthan2p} extends to the bicomplex-poly-Hardy spaces. 

\begin{prop}
For $0 < p < \infty$ and $n$ a positive integer, every $f \in P^{n,p}(D,\bc)$ is an element of $L^m(D,\bc)$, for all $0 < m < 2p$. 
\end{prop}

\begin{proof}
 By Theorem \ref{thm: polyhardyrep}, every $f \in P^{n,p}(D,\bc)$ is representable as $f = \sum_{k = 0}^{n-1} \widehat{z^*}^k f_k$, where $f_k \in H^p(D,\bc)$, for every $k$. Since $f_k \in H^p(D,\bc)$, it follows by Proposition \ref{Prop: bchardymlessthan2p} that $f_k \in L^m(D,\bc)$, for every $0 < m < 2p$. Observe that 
 \begin{align*}
    \iint_D ||f(z)||_{\bc}^m \,dx\,dy 
    &= \iint_D \left|\left|\sum_{k = 0}^{n-1} \widehat{z^*}^k f_k(z)\right|\right|_{\bc}^m \,dx\,dy \\
    &\leq C\sum_{k = 0}^{n-1}\iint_D \left|\left|  f_k(z)\right|\right|_{\bc}^m \,dx\,dy < \infty,
 \end{align*}
 where $C$ is a constant that depends on $n$ and $m$. Thus, $f \in L^m(D,\bc)$. 
 
\end{proof}

\subsection{Boundary Behavior}

In \cite{polyhardy}, the complex poly-Hardy spaces were shown to exhibit the boundary behavior of functions in $H^p(D)$. In the bicomplex poly-Hardy setting, we show that this boundary behavior is recovered with an appeal to the generalized Hardy classes, as was needed for the $\bc$-Vekua-Hardy classes.

\begin{theorem}
    For $n$ a positive integer, $0 < p < \infty$, and $w_k \in L^{q_k}(D,\bc)$, $0 \leq k \leq n-1$, such that $p$ and $q_k$ satisfy $q_k>2$ or $1 < q_k \leq 2$ and $p < \frac{q_k}{2-q_k}$, for all $k$, every $f \in P^{n,p}(D,\bc)$ with representation
    \[
            f = \sum_{k = 0}^{n-1} \widehat{z^*}^k \varphi_k
    \]
    such that $\varphi_k^+ \in H^p_{w_k}(D)$, for each $k$, has a nontangential limit $f_{nt} \in L^p(\p D, \bc)$, and 
    \[
        \lim_{r \nearrow 1} \int_0^{2\pi} ||f_{nt}(e^{i\theta}) - f(re^{i\theta}) ||^p_{\bc} \, d\theta  = 0. 
    \]  
\end{theorem}

\begin{proof}
If $\varphi_k^+ \in H^p_{w_k}(D)$ and $w_k \in L^{q_k}(D)$, where $q_k>2$ or $1 < q_k \leq 2$ and $p < \frac{q_k}{2-q_k}$, then, by Theorem \ref{bcholobvcon}, $\varphi_k$ has a nontangential boundary value $\varphi_{k,nt} \in L^p(\p D, \bc)$ and 
\[
    \lim_{r \nearrow 1} \int_0^{2\pi} ||\varphi_{k,nt}(e^{i\theta}) - \varphi_k(re^{i\theta})||_{\bc}^p \,d\theta = 0.
\]
Since each $\varphi_k$ has a nontangential boundary value $\varphi_{k,nt}$ and $f = \sum_{k = 0}^{n-1} \widehat{z^*}^k \varphi_k$, it follows that $f_{nt}$ exists and is 
\[
    f_{nt} = \sum_{k = 0}^{n-1} \widehat{(e^{i\theta})^*}^k \varphi_{k,nt}.
\]
Observe that 
\begin{align*}
    \int_0^{2\pi} ||f_{nt}(e^{i\theta})||_{\bc}^p\,d\theta
    &= \int_0^{2\pi} \left|\left|\sum_{k = 0}^{n-1} \widehat{(e^{i\theta})^*}^k \varphi_{k,nt}(e^{i\theta})\right|\right|_{\bc}^p\,d\theta\\
    &\leq \sum_{k = 0}^{n-1} C_p \int_0^{2\pi} \left|\left|\widehat{(e^{i\theta})^*}^k \varphi_{k,nt}(e^{i\theta})\right|\right|_{\bc}^p\,d\theta \\
    &\leq \sum_{k = 0}^{n-1} C_p \, 2^{p/2}\int_0^{2\pi} \left|\left| \varphi_{k,nt}(e^{i\theta})\right|\right|_{\bc}^p\,d\theta < \infty,
\end{align*}
where $C_p$ is a constant that depends on only $p$ and $n$. Hence, $f_{nt} \in L^p(\p D, \bc)$. Now, for $r \in (0,1)$, we have
\begin{align*}
    &\int_0^{2\pi} \left|\left| f_{nt}(e^{i\theta}) - f(re^{i\theta})\right|\right|_{\bc}^p \,d\theta \\
    &= \int_0^{2\pi} \left|\left| \sum_{k = 0}^{n-1} \widehat{(e^{i\theta})^*}^k \varphi_{k,nt}(e^{i\theta}) - \sum_{k = 0}^{n-1} \widehat{(re^{i\theta})^*}^k \varphi_k(re^{i\theta}) \right|\right|_{\bc}^p \,d\theta\\
    &\leq C_p \int_0^{2\pi} \left|\left| \sum_{k = 0}^{n-1} \widehat{(e^{i\theta})^*}^k \varphi_{k,nt}(e^{i\theta}) - \sum_{k = 0}^{n-1} \widehat{(e^{i\theta})^*}^k \varphi_k(re^{i\theta}) \right|\right|_{\bc}^p \,d\theta\\
    &\quad\quad +  C_p\int_0^{2\pi} \left|\left| \sum_{k = 0}^{n-1} \widehat{(e^{i\theta})^*}^k \varphi_{k}(re^{i\theta}) - \sum_{k = 0}^{n-1} \widehat{(re^{i\theta})^*}^k \varphi_k(re^{i\theta}) \right|\right|_{\bc}^p \,d\theta \\
    &= C_p \sum_{k = 0}^{n-1}\int_0^{2\pi} \left|\left|  \widehat{(e^{i\theta})^*}^k (\varphi_{k,nt}(e^{i\theta}) -  \varphi_k(re^{i\theta}) )\right|\right|_{\bc}^p \,d\theta\\
    &\quad\quad +  C_p\sum_{k = 0}^{n-1}\int_0^{2\pi} \left|\left|  (\widehat{(e^{i\theta})^*}^k - \widehat{(re^{i\theta})^*}^k )\varphi_k(re^{i\theta}) \right|\right|_{\bc}^p \,d\theta \\
    &\leq C_p \sum_{k = 0}^{n-1}\int_0^{2\pi} \left|\left|  \varphi_{k,nt}(e^{i\theta}) -  \varphi_k(re^{i\theta} )\right|\right|_{\bc}^p \,d\theta +  C_p\sum_{k = 0}^{n-1}\int_0^{2\pi} \left|\left|  (\widehat{(e^{i\theta})^*}^k - \widehat{(re^{i\theta})^*}^k )\varphi_k(re^{i\theta}) \right|\right|_{\bc}^p \,d\theta,
\end{align*}
where $C_p$ is a constant that depends only on $p$ and is not necessarily the same from line to line. Observe that, for every $k$, 
\begin{align*}
(\widehat{(e^{i\theta})^*}^k - \widehat{(re^{i\theta})^*}^k )\varphi_k(re^{i\theta})
&= (1-r)^k e^{-jk\theta} (p^+ \varphi_k^+(re^{i\theta}) + p^- \varphi_k^-(re^{i\theta}))
\end{align*}
By Lebesgue's Dominated Convergence Theorem, 
\begin{align*}
&\lim_{r \nearrow 1}\int_0^{2\pi} \left|\left|  (\widehat{(e^{i\theta})^*}^k - \widehat{(re^{i\theta})^*}^k )\varphi_k(re^{i\theta}) \right|\right|_{\bc}^p \,d\theta \\
&= \lim_{r \nearrow 1}\int_0^{2\pi} \left|\left| (1-r)^k e^{-jk\theta} (p^+ \varphi_k^+(re^{i\theta}) + p^- \varphi_k^-(re^{i\theta})) \right|\right|_{\bc}^p \,d\theta\\
&\leq C_{p,k} \left( \lim_{r \nearrow 1}\int_0^{2\pi} | (1-r)^{k} |^p\, | \varphi_k^+(re^{i\theta})|^p \,d\theta + \lim_{r \nearrow 1}\int_0^{2\pi} | (1-r)^{k} |^p\,|\varphi_k^-(re^{i\theta})) |^p \,d\theta\right) = 0,
\end{align*}
for every $k$. Since 
\[
\lim_{r \nearrow 1}\int_0^{2\pi} \left|\left|  \varphi_{k,nt}(e^{i\theta}) -  \varphi_k(re^{i\theta} )\right|\right|_{\bc}^p \,d\theta = 0,
\]
for every $k$, by Theorem \ref{bcholobvcon}, and 
\[
 \lim_{r \nearrow 1} \int_0^{2\pi} \left|\left|  (\widehat{(e^{i\theta})^*}^k - \widehat{(re^{i\theta})^*}^k )\varphi_k(re^{i\theta}) \right|\right|_{\bc}^p \,d\theta = 0,
\]
where $C_{p,k}$ is a constant that depends on only $p$ and $k$, for every $k$, it follows that 
\begin{align*}
&\lim_{r \nearrow 1} C_p\left[ \sum_{k = 0}^{n-1}\int_0^{2\pi} \left|\left|  \varphi_{k,nt}(e^{i\theta}) -  \varphi_k(re^{i\theta} )\right|\right|_{\bc}^p \,d\theta +  \sum_{k = 0}^{n-1}\int_0^{2\pi} \left|\left|  (\widehat{(e^{i\theta})^*}^k - \widehat{(re^{i\theta})^*}^k )\varphi_k(re^{i\theta}) \right|\right|_{\bc}^p \,d\theta \right]\\
&= 0.
\end{align*}
Thus, 
\[
\lim_{r \nearrow 1} \int_0^{2\pi} \left|\left| f_{nt}(e^{i\theta}) - f(re^{i\theta})\right|\right|_{\bc}^p \,d\theta = 0.
\]

\end{proof}

\section{Bicomplex HOIV-Hardy Spaces}\label{BCHOIVHardyspaces}

In this final section, we combine the previous results concerning $\bc$-poly-Hardy and $\bc$-Vekua-Hardy classes to produce a class of functions that solve higher-order variants of the $\bc$-Vekua equation and exhibit many of the properties associated with a Hardy-type class of functions. This continues the work initiated by the author and B. B. Delgado in \cite{WBD} in the complex setting.

\subsection{Definitions}

\begin{deff}
    Let $0 < p < \infty$, $n$ a positive integer, and $A,B:D \to \bc$. We define the class of bicomplex-HOIV (\textit{higher-order iterated Vekua}) functions $V^n_{A,B}(D,\bc)$ to be those functions $f: D \to \bc$ that solve the bicomplex HOIV-equation
    \[
        (\dbar - A- BC(\cdot))^n f = 0,
    \]
    where $C(\cdot)$ is the operator that applies bicomplex conjugation.
\end{deff}

\begin{deff}
    Let $0 < p < \infty$, $n$ a positive integer, and $A,B \in W^{n-1,q}(D,\bc)$, $q>2$. We define the class of bicomplex HOIV-Hardy functions $V^{n,p}_{A,B}(D,\bc)$ to be those functions $f \in V^n_{A,B}(D,\bc)$ such that 
    \[
        \sum_{k = 0}^{n-1} ||(\dbar - A- BC(\cdot))^k f||_{H^p_{\bc}} < \infty.
    \]
\end{deff}

\subsection{Bicomplex Meta-Hardy Spaces}

    \subsubsection{Definition}

In \cite{WB3}, the author showed that functions which solve the iterated Vekua-type equation
\[
    \left( \frac{\p}{\p z^*} - A \right)^n w = 0
\]
exhibit similarities to the complex polyanalytic functions. This is the special case of the complex HOIV-equation where $B\equiv 0$. In \cite{WB3}, these functions were called meta-analytic functions to associate them with the previously studied special case where $A \in \C$, see \cite{metahardy}, or $A \in Hol(D)$, see \cite{itvek}. Here we show that the connection between polyanalytic and meta-analytic persists into the setting of $\bc$-valued functions.  

\begin{theorem}\label{Thm: bcmetarep}
    Let $n$ be a positive integer and $A \in W^{n-1,q}(D,\bc)$, $q>2$. Every $f \in V^{n}_{A,0}(D,\bc)$ is representable as 
    \[
        f = \sum_{k = 0}^{n-1} \widehat{z^*}^k \varphi_k,
    \]
    where $\varphi_k \in H_{A,0}(D,\bc)$, for each $k$. 
\end{theorem}

\begin{proof}
Observe that functions of the form $f(z) = \sum_{k = 0}^{n-1} \widehat{z^*}^k \varphi_k(z)$, where $\varphi_k \in H_{A,0}(D,\bc)$, for every $k$, are elements of $V^n_{A,0}(D,\bc)$ by direct computation. 

We prove the other direction by induction. The result is clear for $n = 1$. Suppose that, for every $n$ satisfying $1 \leq n \leq m-1$, that if $f \in V^n_{A,0}(D, \bc)$, then there exists a collection $\{\varphi_k\}_{k=0}^{n-1}$ such that $\varphi_k\in H_{A,0}(D,
\bc)$, for every $k$, and
\[
    f = \sum_{k = 0}^{n-1} \widehat{z^*}^k \varphi_k
\]
Let $g \in V^m_{A,0}(D, \bc)$. So, $(\dbar-A) g \in V^{m-1}_{A,0}(D, \bc)$, and there exists a collection $\{\varphi_k\}_{k=0}^{m-2}$ such that $\varphi_k \in H_{A,0}(D,\bc)$, for every $k$, and 
\[
    (\dbar-A) g = \sum_{k=0}^{m-2} \widehat{z^*}^k \varphi_k.
\]
Observe that 
\begin{align*}
    &(\dbar-A)\left( g - \sum_{k=0}^{m-2} \frac{1}{k+1}\widehat{z^*} ^{k+1} \varphi_k\right) \\  
    &= \left(\dbar -A\right)g - \dbar \left( \sum_{k=0}^{m-2} \frac{1}{k+1}\widehat{z^*} ^{k+1} \varphi_k\right) + A  \left( \sum_{k=0}^{m-2} \frac{1}{k+1}\widehat{z^*} ^{k+1} \varphi_k\right)\\
    &= \sum_{k=0}^{m-2} \widehat{z^*}^k \varphi_k - \sum_{k=0}^{m-2} \widehat{z^*}^k \varphi_k  - \sum_{k=0}^{m-2} \frac{1}{k+1}\widehat{z^*} ^{k+1} A\varphi_k +    \sum_{k=0}^{m-2} \frac{1}{k+1}\widehat{z^*} ^{k+1} A\varphi_k \\
    &= 0.
\end{align*}
Thus, $g - \sum_{k=0}^{m-2} \frac{1}{k+1}\widehat{z^*} ^{k+1} \varphi_k \in H_{A,0}(D,\bc)$. Let 
\[
\gamma = g - \sum_{k=0}^{m-2} \frac{1}{k+1}\widehat{z^*} ^{k+1} \varphi_k.
\]
Let $\phi_0 := \gamma$ and $\phi_{k+1} := \frac{1}{k+1} \varphi_k$, for every $0 \leq k \leq m-2$. Then we rearrange to have
\[
    g = \sum_{j = 0}^{m-1} \widehat{z^*}^j \phi_j.
\]

\end{proof}

    \subsubsection{Representation}

    In \cite{WB3}, the author showed that functions in Hardy spaces of meta-analytic function are not only representable as a polynomial in $z^*$ with generalized analytic function coefficients but that the coefficients are elements of the corresponding Vekua-Hardy space. This behavior persists for bicomplex-valued functions.

    \begin{theorem}\label{Thm: bcmetahardyrep}
         For $0 < p < \infty$, $n$ a positive integer, and $A \in W^{n-1, q}(D,\bc)$, $q>2$, a function $f$ is in $V^{n,p}_{A,0}(D,\bc)$ with representation 
        \[
            f = \sum_{k = 0}^{n-1}  \widehat{z^*}^k \varphi_k
        \]
        if and only if $\varphi_k \in H^p_{A,0}(D,\bc)$, for every $k$.
    \end{theorem}

    \begin{proof}
    First, suppose that $f \in V^{n,p}_{A,0}(D,\bc)$. By Theorem \ref{Thm: bcmetarep}, there exists $\varphi_{k} \in H_{A,0}(D,\bc)$ such that
    \[
        f = \sum_{k = 0}^{n-1}  \widehat{z^*}^k \varphi_k.
    \]
    Observe that 
    \[
        \varphi_k = \frac{1}{k!}\sum_{j=0}^{n-1-k} \frac{(-1)^{j}}{j!} \widehat{z^*}^j (\dbar - A)^{k+j}f,
    \]
    for every $k$. For this formula in the complex setting, see Remark 3.2 of \cite{WBD} (or Theorem 1 of \cite{itvek} or Theorem 4.1 of \cite{itvekbvp}). Now, we have
    \begin{align*}
        &\sup_{0< r < 1} \int_0^{2\pi} ||\varphi_k(re^{i\theta})||_{\bc}^p\,d\theta\\
        &= \sup_{0< r < 1} \int_0^{2\pi} \left|\left|\frac{1}{k!}\sum_{j=0}^{n-1-k} \frac{(-1)^{j}}{j!} \widehat{(re^{i\theta})^*}^j (\dbar - A)^{k+j}f(re^{i\theta})\right|\right|_{\bc}^p\,d\theta \\
        &\leq C_{n,p} \sum_{j=0}^{n-1-k}\sup_{0< r < 1} \int_0^{2\pi} \left|\left|(\dbar - A)^{k+j}f(re^{i\theta})\right|\right|_{\bc}^p\,d\theta< \infty,
    \end{align*}
    where $C_{n,p}$ is a constant that depends only on $n$ and $p$. Hence, $\varphi_k \in H^p_{A,0}(D)$. 

    Next, suppose that $\varphi_k \in H^p_{A,0}(D)$, for every $k$. Since 
    \[
        (\dbar - A)^k f = \sum_{j=k}^{n-1} C_{kj} \widehat{z^*}^{j-k} \varphi_j,
    \]
    for every $0 \leq k \leq n-1$, where $C_{kj}= k!$$ {j}\choose{j-k}$, it follows that 
    \begin{align*}
        &\sup_{0 < r < 1} \int_0^{2\pi} ||(\dbar - A)^k f(re^{i\theta})||_{\bc}^p \,d\theta \\
        &= \sup_{0 < r < 1} \int_0^{2\pi} \left|\left|\sum_{j=k}^{n-1} C_{kj} \widehat{(re^{i\theta})^*}^{j-k} \varphi_j(re^{i\theta})\right|\right|_{\bc}^p \,d\theta \\
        &\leq C_{n,p} \sum_{j=k}^{n-1} \sup_{0 < r < 1} \int_0^{2\pi} \left|\left|\varphi_j(re^{i\theta})\right|\right|_{\bc}^p \,d\theta < \infty,
    \end{align*}
    where $C_{n,p}$ is a constant that depends on only $n$ and $p$. 

    \end{proof}

    Using the representation above, we extend Theorem \ref{inBCVekHardypinLebm} to the bicomplex-meta-Hardy spaces. 

\begin{theorem}
Let $p$ be a positive real number, $n$ a positive integer, and $A\in W^{n-1,q}(D)$, $q>2$. Every function $w \in V^{n,p}_{A,0}(D,\bc)$ is an element of $L^m(D,\bc)$, for every $0 < m < 2p$. 
\end{theorem}

\begin{proof}
 By Theorem \ref{Thm: bcmetahardyrep}, every $f \in V^{n,p}_{A,0}(D,\bc)$ is representable as $f = \sum_{k = 0}^{n-1}  \widehat{z^*}^k f_k$, where $f_k \in H^p_{A,0}(D,\bc)$. By Theorem \ref{inBCVekHardypinLebm}, $f_k \in L^m(D,\bc)$, for every $0 < m < 2p$ and every $k$. Hence, we have
 \begin{align*}
    \iint_D ||f(z)||_{\bc}^m \,dx\,dy 
    &= \iint_D \left|\left|\sum_{k = 0}^{n-1} \widehat{z^*}^k f_k(z)\right|\right|_{\bc}^m \,dx\,dy \\
    &\leq C\sum_{k = 0}^{n-1}\iint_D \left|\left|  f_k(z)\right|\right|_{\bc}^m \,dx\,dy < \infty,
 \end{align*}
 where $C$ is a constant that depends on $n$ and $m$. Therefore, $f \in L^m(D,\bc)$. 
\end{proof}

    \subsubsection{Boundary Behavior}

    Next, we show that $\bc$-meta-Hardy classes also share the boundary behavior associated with Hardy classes.

    \begin{theorem}
        For $0 < p < \infty$, $n$ a positive integer, and $A \in W^{n-1, q}(D,\bc)$, $q>2$, every $f =\sum_{k=0}^{n-1} \widehat{z^*}^k \varphi_k \in V^{n,p}_{A,0}(D,\bc)$ such that $\varphi^+_k \in H^p_{w_k}(D)$, for $w_k \in L^{q_k}(D)$ with $q_k>2$ or $1 < q_k \leq 2$ and $p< \frac{q_k}{2-q_k}$, for every $k$, has a nontangential limit $f_{nt} \in L^p(D, \bc)$, and 
        \[
            \lim_{r \nearrow 1} \int_0^{2\pi} ||f_{nt}(e^{i\theta}) - f(re^{i\theta}) ||^p_{\bc} \, d\theta  = 0. 
        \]  
    \end{theorem}

    \begin{proof}
        Let $f \in V^{n,p}_{A,0}(D,\bc)$. By Theorem \ref{Thm: bcmetahardyrep}, 
        \[
            f = \sum_{k = 0}^{n-1}  \widehat{z^*}^k \varphi_k,
        \]
        where $\varphi_k \in H^p_{A,0}(D,\bc)$, for every $k$. If $\varphi_k^+ \in H^p_{w_k}(D)$, where $w_k \in L^{q_k}(D)$ with $q_k>2$ or $1 < q_k \leq 2$ and $p< \frac{q_k}{2-q_k}$, for every $k$, it follows that $\varphi_{k,nt}$ exists, is in $L^p(\p D)$, and 
        \begin{equation}\label{convfork}
            \lim_{r \nearrow 1} \int_0^{2\pi} ||\varphi_{k,nt}(e^{i\theta}) - \varphi_k(re^{i\theta}) ||^p_{\bc} \, d\theta  = 0. 
        \end{equation}
        Since $f = \sum_{k = 0}^{n-1}  \widehat{z^*}^k \varphi_k$, it follows that 
        \[
            f_{nt} = \sum_{k = 0}^{n-1}  \widehat{(e^{i\theta})^*}^k \varphi_{k,nt}
        \]
        exists, and, for $r \in (0,1)$, we have
        \begin{align*}
            &\int_0^{2\pi} ||f_{nt}(e^{i\theta}) - f(re^{i\theta}) ||^p_{\bc} \, d\theta \\
            &= \int_0^{2\pi} ||\sum_{k = 0}^{n-1}  \widehat{(e^{i\theta})^*}^k \varphi_{k,nt}(e^{i\theta}) - \sum_{k = 0}^{n-1}  \widehat{(re^{i\theta})^*}^k \varphi_k(re^{i\theta}) ||^p_{\bc} \, d\theta \\
            &\leq C_p \int_0^{2\pi} ||\sum_{k = 0}^{n-1}  \widehat{(e^{i\theta})^*}^k \varphi_{k,nt}(e^{i\theta}) - \sum_{k = 0}^{n-1}  \widehat{(e^{i\theta})^*}^k \varphi_k(re^{i\theta}) ||^p_{\bc} \, d\theta \\
            &\quad\quad + C_p\int_0^{2\pi} ||\sum_{k = 0}^{n-1}  \widehat{(e^{i\theta})^*}^k \varphi_k(re^{i\theta}) - \sum_{k = 0}^{n-1}  \widehat{(re^{i\theta})^*}^k \varphi_k(re^{i\theta}) ||^p_{\bc} \, d\theta \\
            &\leq C_{p,n} \sum_{k = 0}^{n-1} \int_0^{2\pi} ||  \widehat{(e^{i\theta})^*}^k (\varphi_{k,nt}(e^{i\theta}) - \varphi_k(re^{i\theta})) ||^p_{\bc} \, d\theta \\
            &\quad\quad + C_{p,n} \sum_{k = 0}^{n-1} \int_0^{2\pi}  ||  (\widehat{(e^{i\theta})^*}^k -  \widehat{(re^{i\theta})^*}^k )\varphi_k(re^{i\theta}) ||^p_{\bc} \, d\theta \\
            &\leq \widetilde{C_{p,n}} \sum_{k = 0}^{n-1} \int_0^{2\pi} ||  \varphi_{k,nt}(e^{i\theta}) - \varphi_k(re^{i\theta}) ||^p_{\bc} \, d\theta \\
            &\quad\quad + C_{p,n} \sum_{k = 0}^{n-1} \int_0^{2\pi}  ||  (\widehat{(e^{i\theta})^*}^k -  \widehat{(re^{i\theta})^*}^k )\varphi_k(re^{i\theta}) ||^p_{\bc} \, d\theta ,
        \end{align*}
        where $C_{p,n}$ and $\widetilde{C_{p,n}}$ are constants that depend on only $p$ and $n$. Observe that 
        \begin{align*}
            & (\widehat{(e^{i\theta})^*}^k -  \widehat{(re^{i\theta})^*}^k )\varphi_k(re^{i\theta}) \\
            &= (e^{-jk\theta} - r^k e^{-jk\theta}) (p^+ \varphi_k^+(re^{i\theta}) + p^- \varphi_k^-(re^{i\theta})) \\
            &= (p^+ e^{ik\theta} + p^- e^{-ik\theta} - r^k (p^+ e^{ik\theta} + p^- e^{-ik\theta})) (p^+ \varphi_k^+(re^{i\theta}) + p^- \varphi_k^-(re^{i\theta})) \\
            &= (p^+ (1-r^k)e^{ik\theta} + p^- (1-r^k)e^{-ik\theta} ) (p^+ \varphi_k^+(re^{i\theta}) + p^- \varphi_k^-(re^{i\theta})) \\
            &= p^+ (1-r^k)e^{ik\theta}\varphi_k^+(re^{i\theta}) + p^- (1-r^k)e^{-ik\theta} \varphi_k^-(re^{i\theta}) .
        \end{align*}

        So, 
        \begin{align*}
&           \widetilde{C_{p,n}} \sum_{k = 0}^{n-1} \int_0^{2\pi}                ||  \varphi_{k,nt}(e^{i\theta}) -                                   \varphi_k(re^{i\theta}) ||^p_{\bc} \, d\theta \\
            &\quad\quad + C_{p,n} \sum_{k = 0}^{n-1} \int_0^{2\pi}  ||  (\widehat{(e^{i\theta})^*}^k -  \widehat{(re^{i\theta})^*}^k )\varphi_k(re^{i\theta}) ||^p_{\bc} \, d\theta\\
            &= \widetilde{C_{p,n}} \sum_{k = 0}^{n-1} \int_0^{2\pi} ||  \varphi_{k,nt}(e^{i\theta}) - \varphi_k(re^{i\theta}) ||^p_{\bc} \, d\theta \\
            &\quad\quad + C_{p,n} \sum_{k = 0}^{n-1} \int_0^{2\pi}  || p^+ (1-r^k)e^{ik\theta}\varphi_k^+(re^{i\theta}) + p^- (1-r^k)e^{-ik\theta} \varphi_k^-(re^{i\theta}) ||^p_{\bc} \, d\theta\\
            &\leq \widetilde{C_{p,n}} \sum_{k = 0}^{n-1} \int_0^{2\pi} ||  \varphi_{k,nt}(e^{i\theta}) - \varphi_k(re^{i\theta}) ||^p_{\bc} \, d\theta \\
            &\quad\quad + C_{p,n} 2^{-p/2}\sum_{k = 0}^{n-1} \int_0^{2\pi}  \left( |(1-r^k)e^{ik\theta}\varphi_k^+(re^{i\theta})| + |(1-r^k)e^{-ik\theta} \varphi_k^-(re^{i\theta}) | \right)^p \, d\theta\\
            &\leq \widetilde{C_{p,n}} \sum_{k = 0}^{n-1} \int_0^{2\pi} ||  \varphi_{k,nt}(e^{i\theta}) - \varphi_k(re^{i\theta}) ||^p_{\bc} \, d\theta \\
            &\quad\quad + \widetilde{\widetilde{C_{p,n}}}\sum_{k = 0}^{n-1} \int_0^{2\pi}  |(1-r^k)|^p\, |\varphi_k^+(re^{i\theta})|^p\,d\theta \\
            &\quad\quad + \widetilde{\widetilde{C_{p,n}}}\sum_{k = 0}^{n-1} \int_0^{2\pi} |(1-r^k)|^p \, | \varphi_k^-(re^{i\theta}) |^p  \, d\theta ,
        \end{align*}
        where $C_{p,n}$, $\widetilde{C_{p,n}}$, and $\widetilde{\widetilde{C_{p,n}}}$ are constants that depend on only $p$ and $n$. Observe that, as $r \nearrow 1$, the first sum in the last line of the inequality converges to zero because of (\ref{convfork}), and the second and third sums in the last line of the inequality converge to zero by Lebesgue's Dominated Convergence Theorem. Therefore, 
        \[
            \lim_{r \nearrow 1} \int_0^{2\pi} ||f_{nt}(e^{i\theta}) - f(re^{i\theta}) ||^p_{\bc} \, d\theta  = 0. 
        \]  
    \end{proof}

We also show that the result of Theorem \ref{bcvekhardybdinsfuncins} extends to the bicomplex-meta-Hardy spaces. This is a bicomplexification of Corollary 5.2 in \cite{WBD}.

\begin{theorem}
    For $0 < p < \infty$, $n$ a positive integer, and $A \in W^{n-1, q}(D,\bc)$, $q>2$, every function $f \in V^{n,p}_{A,0}(D,\bc)$ with representation 
        \[
            f = \sum_{k = 0}^{n-1}  \widehat{z^*}^k f_k
        \]
        such that $f_{k,nt} \in L^{s_k}(\p D, \bc)$, where $s_k > p$, for every $k$, is in $V^{n,s}_{A,0}(D,\bc)$, where $s:=\min_{k}\{s_k\}$.
\end{theorem}

\begin{proof}
By Theorem \ref{Thm: bcmetahardyrep}, if $f \in V^{n,p}_{A,0}(D,\bc)$, then $f = \sum_{k = 0}^{n-1}  \widehat{z^*}^k f_k$ and $f_k \in H^p_{A,0}(D,\bc)$, for every $k$. Since every $f_k \in H^p_{A,0}(D,\bc)$ and each $f_{k,nt} \in L^{s_k}(\p D,\bc)$, for some $s_k > p$, it follows, by Theorem \ref{bcvekhardybdinsfuncins}, that $f_k \in H^{s_k}_{A,0}(D,\bc)$, for each $k$. Hence, every $f_k \in H^s_{A,0}(D,\bc)$, for every $k$, where $s:= \min_{k}\{s_k\}$. By Theorem \ref{Thm: bcmetahardyrep}, since every $f_k \in H^s_{A,0}(D,\bc)$, it follows that $f \in V^{n,s}_{A,0}(D,\bc)$. 
\end{proof}

\subsection{Representation}

We now return to the general case of solutions to HOIV-equations. First, we prove that in the general setting the $\bc$-HOIV-Hardy class functions have a polynomial representation. Unlike in the special case, these polynomials will be in terms of $\widehat{z}+\widehat{z^*} = 2\,\sca{z}$.

\begin{theorem}\label{Thm: bchoivrep}
    For $n$ a positive integer and $A,B \in W^{n-1,q}(D,\bc)$, $q>2$, every $f \in V^{n}_{A,B}(D,\bc)$ has the form 
    \[
        f = \sum_{k = 0}^{n-1} (\widehat{z} + \widehat{z^*})^k \varphi_k,
    \]
    where $\varphi_k \in \hab$, for every $k$. 
\end{theorem}

\begin{proof}
Observe that functions of the form $f(z) = \sum_{k = 0}^{n-1} (\widehat{z} + \widehat{z^*})^k \varphi_k(z)$, where $\varphi_k \in \hab$, for every $k$, are elements of $V^n_{A,B}(D,\bc)$ by direct computation. 

We prove the other direction by induction. The result is clear for $n = 1$. Suppose that, for every $n$ satisfying $1 \leq n \leq m-1$, that if $f \in V^n_{A,B}(D, \bc)$, then there exists a collection $\{\varphi_k\}_{k=0}^{n-1}$ such that $\varphi_k\in \hab$, for every $k$, and
\[
    f = \sum_{k = 0}^{n-1} (\widehat{z} + \widehat{z^*})^k \varphi_k
\]
Let $g \in V^m_{A,B}(D, \bc)$. So, $(\dbar-A-BC(\cdot)) g \in V^{m-1}_{A,B}(D, \bc)$, and there exists a collection $\{\varphi_k\}_{k=0}^{m-2}$ such that $\varphi_k \in H_{A,B}(D,\bc)$, for every $k$, and 
\[
    (\dbar-A-BC(\cdot)) g = \sum_{k=0}^{m-2} (\widehat{z} + \widehat{z^*})^k \varphi_k.
\]
Observe that 
\begin{align*}
    &(\dbar-A- BC(\cdot))\left( g - \sum_{k=0}^{m-2} \frac{1}{k+1}(\widehat{z} + \widehat{z^*})^{k+1} \varphi_k\right) \\  
    &= \left(\dbar -A- BC(\cdot)\right)g - (\dbar-A- BC(\cdot)) \sum_{k=0}^{m-2} \frac{1}{k+1}(\widehat{z} + \widehat{z^*})^{k+1} \varphi_k \\
    &= \sum_{k=0}^{m-2} (\widehat{z} + \widehat{z^*})^k \varphi_k j- \sum_{k=0}^{m-2} (\widehat{z} + \widehat{z^*})^k \varphi_k- \sum_{k=0}^{m-2} \frac{1}{k+1}(\widehat{z} + \widehat{z^*})^{k+1} (A\varphi_k + BC(\varphi_k)) \\
    &\quad\quad + A \sum_{k=0}^{m-2} \frac{1}{k+1}(\widehat{z} + \widehat{z^*})^{k+1} \varphi_k + BC(\sum_{k=0}^{m-2} \frac{1}{k+1}(\widehat{z} + \widehat{z^*})^{k+1} \varphi_k)\\
    &=- \sum_{k=0}^{m-2} \frac{1}{k+1}(\widehat{z} + \widehat{z^*})^{k+1} (A\varphi_k + BC(\varphi_k)) \\
    &\quad\quad + \sum_{k=0}^{m-2} \frac{1}{k+1}(\widehat{z} + \widehat{z^*})^{k+1} A\varphi_k + \sum_{k=0}^{m-2} \frac{1}{k+1}(\widehat{z} + \widehat{z^*})^{k+1} BC(\varphi_k)\\
    &= 0.
\end{align*}

Thus, $g - \sum_{k=0}^{m-2} \frac{1}{k+1}(\widehat{z} + \widehat{z^*})^{k+1} \varphi_k \in H_{A,B}(D,\bc)$. Let 
\[
\gamma = g - \sum_{k=0}^{m-2} \frac{1}{k+1}(\widehat{z} + \widehat{z^*}) ^{k+1} \varphi_k.
\]
Let $\phi_0 := \gamma$ and $\phi_{k+1} := \frac{1}{k+1} \varphi_k$, for every $0 \leq k \leq m-2$. Then we rearrange to have
\[
    g = \sum_{j = 0}^{m-1} (\widehat{z} + \widehat{z^*})^j \phi_j.
\]

\end{proof}

Similarly to Theorem \ref{Thm: bcmetahardyrep} and extending Theorem 4.1 of \cite{WBD} to the bicomplex setting, we show that if a function is in a $\bc$-HOIV-Hardy class, then the coefficients of the polynomial representation are in the associated $\bc$-Vekua-Hardy space. 

\begin{theorem}\label{Thm: bchoivhardyrep}
    For $0 < p < \infty$, $n$ a positive integer, and $A,B \in W^{n-1, q}(D,\bc)$, $q>2$, a function $f$ is in $V^{n,p}_{A,B}(D,\bc)$ with representation 
    \[
        f = \sum_{k = 0}^{n-1} (\widehat{z} + \widehat{z^*})^k \varphi_k
    \]
    if and only if $\varphi_k \in \hpab$, for every $k$. 
\end{theorem}

\begin{proof}
First, suppose that $f \in V^{n,p}_{A,B}(D,\bc)$. By Theorem \ref{Thm: bchoivrep}, 
\[
    f = \sum_{k = 0}^{n-1} (\widehat{z} + \widehat{z^*})^k \varphi_k,
\]
where
\[
    \dbar\varphi_k = A\varphi_k + B\overline{\varphi_k},
\]
for every $k$. Observe that each $\varphi_k$ is representable as
\[
    \varphi_k = \frac{1}{k!}\sum_{j=0}^{n-1-k} \frac{(-1)^{j}}{j!} (\widehat{z}+\widehat{z^*})^j (\dbar - A-BC(\cdot))^{k+j}f,
\]
    for every $k$. For this formula in the complex setting, see Remark 3.2 of \cite{WBD} (or Theorem 1 of \cite{itvek} or Theorem 4.1 of \cite{itvekbvp}). Using this representation, we can see that, for $r \in (0,1)$, 
\begin{align*}
    &\int_0^{2\pi} ||\varphi_k(re^{i\theta})||_{\bc}^p \,d\theta\\
    &= \int_0^{2\pi} ||\frac{1}{k!}\sum_{j=0}^{n-1-k} \frac{(-1)^{j}}{j!} (\widehat{re^{i\theta}}+\widehat{(re^{i\theta})^*})^j (\dbar - A-BC(\cdot))^{k+j}f(re^{i\theta})||_{\bc}^p \,d\theta\\
    &\leq C_{n,k,p} \sum_{j=0}^{n-1-k}  \int_0^{2\pi} ||(\widehat{re^{i\theta}}+\widehat{(re^{i\theta})^*})^j (\dbar - A-BC(\cdot))^{k+j}f(re^{i\theta})||_{\bc}^p \,d\theta\\
    &\leq C_{n,p} \sum_{j=0}^{n-1-k} ||(\dbar - A-BC(\cdot))^{k+j}f(re^{i\theta})||_{H^p}^p <\infty,
\end{align*}
where $C_{n,k,p}$ and $C_{n,p}$ are constants that depend on only $n$, $k$, and $p$. Hence, 
\[
    \sup_{0 < r < 1} \int_0^{2\pi} ||\varphi_k(re^{i\theta})||_{\bc}^p \,d\theta < \infty,
\]
for every $k$, and $\varphi_k \in H^p_{A,B}(D, \bc)$, for every $k$.

Now, suppose that $f=\sum_{k = 0}^{n-1} (\widehat{z} + \widehat{z^*})^k \varphi_k$ solves $(\dbar - A - BC(\cdot))^n f = 0$ and $\varphi_k \in H^p_{A,B}(D,\bc)$, for every $k$. Since 
\[
    (\dbar - A- BC(\cdot))^j f = \sum_{k=j}^{n-1} C_{j,n}(\widehat{z} + \widehat{z^*})^{k-j}\varphi_k,
\]
where $C_{j,n}$ is a constant that depends on only $j$ and $k$, for every $0 \leq j \leq n-1$, it follows that, for $r \in (0,1)$, 
\begin{align*}
    &\int_0^{2\pi} ||(\dbar - A- BC(\cdot))^j f(re^{i\theta})||_{\bc}^p \,d\theta\\
    &= \int_0^{2\pi} ||\sum_{k=j}^{n-1} C_{j,n}(\widehat{re^{i\theta}} + \widehat{(re^{i\theta})^*})^{k-j}\varphi_k(re^{i\theta})||_{\bc}^p \,d\theta \\
    &\leq C_{n,p} \sum_{k=j}^{n-1}\int_0^{2\pi}||(\widehat{re^{i\theta}} + \widehat{(re^{i\theta})^*})^{k-j}\varphi_k(re^{i\theta})||_{\bc}^p\\
    &\leq C \sum_{k=j}^{n-1} ||\varphi_k||^{p}_{H^p} < \infty,
\end{align*}
where $C_{n,p}$ and $C$ are constants that depend on only $n$ and $p$. Thus, 
\[
 \sup_{0 < r < 1} \int_0^{2\pi} ||(\dbar - A- BC(\cdot))^j f(re^{i\theta})||_{\bc}^p \,d\theta < \infty,
\]
for every $0 \leq j\leq n-1$, and $f \in V^{n,p}_{A,B}(D,\bc)$. 
\end{proof}

Now, we show that for a specific case of the most general bicomplex-HOIV-Hardy spaces that inclusion in a Hardy space implies inclusion in a range of Lebesgue spaces. This is a bicomplexification of Theorem 4.2 from \cite{WBD}.

\begin{theorem}
For $0 < p < \infty$, $n$ a positive integer, and $A,B \in W^{n-1, q}(D,\bc)$, $q>2$, every function $f \in V^{n,p}_{A,B}(D,\bc)$ with representation $f = \sum_{k = 0}^{n-1} (\widehat{z} + \widehat{z^*})^k \varphi_k$ such that $\dbar \varphi_k \in L^{q_k}(D,\bc)$, where $q_k>2$ or $1 < q_k \leq 2$ and $p < \frac{q_k}{2 - q_k}$, for every $k$, is in $L^m(D,\bc)$, for every $0 < m < 2p$. 
\end{theorem}

\begin{proof}
    By Theorem \ref{Thm: bchoivhardyrep}, if $f \in V^{n,p}_{A,B}(D,\bc)$, then $\sum_{k = 0}^{n-1} (\widehat{z} + \widehat{z^*})^k f_k$, where every 
    $f_k \in H^p_{A,B}(D,\bc)$. Since $f_k \in H^p_{A,B}(D,\bc)$ and $\dbar \varphi_k \in L^{q_k}(D,\bc)$, $q_k>2$ or $1 < q_k \leq 2$ and $p < \frac{q_k}{2 - q_k}$, for every $k$, it follows by Theorem \ref{genbcvekhardyinlebm} that $f_k \in L^m(D,\bc)$, for every $0 < m < 2p$. Thus, 
    \begin{align*}
    \iint_D ||f(z)||_{\bc}^m \,dx\,dy 
    &= \iint_D \left|\left|\sum_{k = 0}^{n-1} (\widehat{z} + \widehat{z^*})^k f_k(z)\right|\right|_{\bc}^m \,dx\,dy \\
    &\leq C\sum_{k = 0}^{n-1}\iint_D \left|\left|  f_k(z)\right|\right|_{\bc}^m \,dx\,dy < \infty,
 \end{align*}
 where $C$ is a constant that depends on $n$ and $m$. Therefore, $f \in L^m(D,\bc)$. 
\end{proof}

\subsection{Boundary Behavior}

Finally, we show that the expected boundary behavior of a Hardy space is recovered in the bicomplex-HOIV-Hardy spaces. The sufficient conditions that comprise the hypotheses of these theorem are taken from the statements of Theorems \ref{Thm: bcvekhardybvcon} and \ref{bcvekhardybvfromidempotent} that guarantee existence of boundary values in the first-order case. The corollary that follows is a bicomplexification of Corollary 5.4 in \cite{WBD} that follows from the combination of the two theorems and Theorem \ref{lonedistbv}.

\begin{theorem}\label{thm: bchoivhardybvcon}
    For $n$ a positive integer, $0 < p < \infty$, and $A, B \in W^{n-1,q}(D,\bc)$, $q>2$, every $f \in V^{n,p}_{A,B}(D,\bc)$ with representation $f = \sum_{k=0}^{n-1} (\widehat{z^*} + \widehat{z})^k \varphi_k$ such that $\varphi_k = h_k + T_\bc(A\varphi_k + B\overline{\varphi_k})$, $\dbar \varphi_k \in L^{q_k}(D,\bc)$, $q_k > 2$, and $h_k^+ \in H^p_{\gamma_k}(D)$, $\gamma_k \in L^{\sigma_k}(D)$, $\sigma_k > 2$ or $1 < \sigma_k \leq 2$ and $p < \frac{\sigma_k}{2-\sigma_k}$, for every $k$, has a nontangential boundary value $f_{nt} \in L^p(\p D,\bc)$ and 
    \[
        \lim_{r \nearrow 1} \int_0^{2\pi} ||f_{nt}(e^{i\theta}) - f(re^{i\theta}) ||^p_{\bc} \, d\theta  = 0. 
    \]
    The result holds for $q_k$ satisfying $1 < q_k \leq 2$ so long as $p$ satisfies $\frac{q_k}{2-q_k}$, for each relevant $k$.
\end{theorem}

\begin{proof}
By Theorem \ref{Thm: bchoivhardyrep}, if $f \in V^{n,p}_{A,B}(D,\bc)$, then $f = \sum_{k = 0}^{n-1} (\widehat{z} + \widehat{z^*})^k \varphi_k$ and $\varphi_k \in H^p_{A,B}(D,\bc)$, for every $k$. By Theorem \ref{bcvekhardyrepgeneral}, if $\dbar \varphi_k \in L^{q_k}(D,\bc)$, $q_k > 2$, $\varphi_k$ is representable as 
\[
        \varphi_k = h_k + T_\bc(A\varphi_k + B\overline{\varphi_k}),
\]
where $h_k \in H^p(D,\bc)$. By Theorem \ref{Thm: bcvekhardybvcon}, if $h_k^+ \in H^p_{\gamma_k}(D)$, $\gamma_k \in L^{\sigma_k}(D)$, $\sigma_k > 2$ or $1 < \sigma_k \leq 2$ and $p < \frac{\sigma_k}{2-\sigma_k}$, then $h_k$ has a nontangential boundary value $\varphi_{k,nt} \in L^p(\p D, \bc)$ and 
\begin{equation}\label{bvconhoivk}
        \lim_{r \nearrow 1} \int_0^{2\pi} ||\varphi_{k,nt}(e^{i\theta}) - \varphi_k(re^{i\theta}) ||^p_{\bc} \, d\theta  = 0. 
    \end{equation}
By our hypotheses, this holds for every $\varphi_k$, $0 \leq k \leq n-1$. So, $f_{nt} = \sum_{k = 0}^{n-1} (\widehat{e^{i\theta}} + \widehat{(e^{i\theta})^*})^k \varphi_{k,nt}$ exists and is in $L^p(\p D, \bc)$. Also, for $r \in (0,1)$, we have
\begin{align*}
    &\int_0^{2\pi} ||f_{nt}(e^{i\theta}) - f(re^{i\theta}) ||^p_{\bc} \, d\theta \\
    &= \int_0^{2\pi} ||\sum_{k = 0}^{n-1} (\widehat{e^{i\theta}} + \widehat{(e^{i\theta})^*})^k \varphi_{k,nt}(e^{i\theta}) - \sum_{k = 0}^{n-1} (\widehat{re^{i\theta}} + \widehat{(re^{i\theta})^*})^k \varphi_k(re^{i\theta}) ||^p_{\bc} \, d\theta\\
    &\leq C_{p} \int_0^{2\pi} ||\sum_{k = 0}^{n-1} (\widehat{e^{i\theta}} + \widehat{(e^{i\theta})^*})^k \varphi_{k,nt}(e^{i\theta}) - \sum_{k = 0}^{n-1} (\widehat{e^{i\theta}} + \widehat{(e^{i\theta})^*})^k \varphi_k(re^{i\theta}) ||^p_{\bc} \, d\theta\\
    &\quad\quad + C_{p} \int_0^{2\pi} ||\sum_{k = 0}^{n-1} (\widehat{e^{i\theta}} + \widehat{(e^{i\theta})^*})^k \varphi_{k}(re^{i\theta}) - \sum_{k = 0}^{n-1} (\widehat{re^{i\theta}} + \widehat{(re^{i\theta})^*})^k \varphi_k(re^{i\theta}) ||^p_{\bc} \, d\theta\\
    &= C_{p,n} \sum_{k = 0}^{n-1} \int_0^{2\pi} || (\widehat{e^{i\theta}} + \widehat{(e^{i\theta})^*})^k (\varphi_{k,nt}(e^{i\theta}) -  \varphi_k(re^{i\theta}) )||^p_{\bc} \, d\theta\\
    &\quad\quad + C_{p,n} \sum_{k = 0}^{n-1} \int_0^{2\pi} || [(\widehat{e^{i\theta}} + \widehat{(e^{i\theta})^*})^k -  (\widehat{re^{i\theta}} + \widehat{(re^{i\theta})^*})^k ]\varphi_k(re^{i\theta}) ||^p_{\bc} \, d\theta\\
    &\leq C_{p,n,k} \sum_{k = 0}^{n-1} \int_0^{2\pi} || \varphi_{k,nt}(e^{i\theta}) -  \varphi_k(re^{i\theta}) ||^p_{\bc} \, d\theta\\
    &\quad\quad + C_{p,n} \sum_{k = 0}^{n-1} \int_0^{2\pi} || [(\widehat{e^{i\theta}} + \widehat{(e^{i\theta})^*})^k -  (\widehat{re^{i\theta}} + \widehat{(re^{i\theta})^*})^k ]\varphi_k(re^{i\theta}) ||^p_{\bc} \, d\theta,
\end{align*}
where $C_p$, $C_{p,n}$, and $C_{p,n,k}$ are all constants that only depend on their subscripts. Observe that 
\begin{align*}
(\widehat{e^{i\theta}} + \widehat{(e^{i\theta})^*})^k -  (\widehat{re^{i\theta}} + \widehat{(re^{i\theta})^*})^k 
&= (e^{j\theta} + e^{-j\theta})^k -  (re^{j\theta} + re^{-j\theta})^k  \\
&= \cos(k\theta) - r^k \cos(k\theta) \\
&= (1-r^k)\cos(k\theta). 
\end{align*}
So, for each $k$, 
\begin{align*}
&\int_0^{2\pi} || [(\widehat{e^{i\theta}} + \widehat{(e^{i\theta})^*})^k -  (\widehat{re^{i\theta}} + \widehat{(re^{i\theta})^*})^k ]\varphi_k(re^{i\theta}) ||^p_{\bc} \, d\theta\\
&=   \int_0^{2\pi} || [(1-r^k)\cos(k\theta) ]\varphi_k(re^{i\theta}) ||^p_{\bc} \, d\theta \\
&\leq 2^{p/2} \int_0^{2\pi} || (1-r^k)\cos(k\theta) ||_{\bc}^p \,||\varphi_k(re^{i\theta}) ||^p_{\bc} \, d\theta \\
&= 2^{p/2} \int_0^{2\pi} | (1-r^k)\cos(k\theta) |^p \,||\varphi_k(re^{i\theta}) ||^p_{\bc} \, d\theta \\
&\leq 2^{p/2} \int_0^{2\pi} (1-r^k)^p \,||\varphi_k(re^{i\theta}) ||^p_{\bc} \, d\theta.
\end{align*}
Hence, 
\begin{align*}
&C_{p,n,k} \sum_{k = 0}^{n-1} \int_0^{2\pi} || \varphi_{k,nt}(e^{i\theta}) -  \varphi_k(re^{i\theta}) ||^p_{\bc} \, d\theta\\
    &\quad\quad + C_{p,n} \sum_{k = 0}^{n-1} \int_0^{2\pi} || [(\widehat{e^{i\theta}} + \widehat{(e^{i\theta})^*})^k -  (\widehat{re^{i\theta}} + \widehat{(re^{i\theta})^*})^k ]\varphi_k(re^{i\theta}) ||^p_{\bc} \, d\theta\\
&\leq C_{p,n,k} \sum_{k = 0}^{n-1} \int_0^{2\pi} || \varphi_{k,nt}(e^{i\theta}) -  \varphi_k(re^{i\theta}) ||^p_{\bc} \, d\theta\\
    &\quad\quad + 2^{p/2} C_{p,n} \sum_{k = 0}^{n-1}  \int_0^{2\pi} (1-r^k)^p \,||\varphi_k(re^{i\theta}) ||^p_{\bc} \, d\theta.
\end{align*}
Observe that the first sum in the last line of the inequality converges to zero as $r \nearrow 1$, by (\ref{bvconhoivk}), and the second sum converges to zero as $r\nearrow 1$, by Lebesgue's Dominated Convergence Theorem. Therefore, 
\[
\lim_{r \nearrow 1} \int_0^{2\pi} ||f_{nt}(e^{i\theta}) - f(re^{i\theta}) ||^p_{\bc} \, d\theta  = 0.
\]

\end{proof}

\begin{theorem}
    For $n$ a positive integer, $0 < p < \infty$, and $A, B \in W^{n-1,q}(D,\bc)$, $q>2$, every $f \in V^{n,p}_{A,B}(D,\bc)$ with representation $f = \sum_{k=0}^{n-1} (\widehat{z^*} + \widehat{z})^k \varphi_k$ such that $\varphi^+_k  \in H^p_{\gamma_k}(D)$, $\gamma_k \in L^{\sigma_k}(D)$, $\sigma_k > 2$ or $1 < \sigma_k \leq 2$ and $p < \frac{\sigma_k}{2-\sigma_k}$, for every $k$, has a nontangential boundary value $f_{nt} \in L^p(\p D,\bc)$ and 
    \[
        \lim_{r \nearrow 1} \int_0^{2\pi} ||f_{nt}(e^{i\theta}) - f(re^{i\theta}) ||^p_{\bc} \, d\theta  = 0. 
    \]
\end{theorem}

\begin{proof}
By Theorem \ref{Thm: bchoivhardyrep}, if $f \in V^{n,p}_{A,B}(D,\bc)$, then $f = \sum_{k = 0}^{n-1} (\widehat{z} + \widehat{z^*})^k \varphi_k$ and $\varphi_k \in H^p_{A,B}(D,\bc)$, for every $k$. By Theorem \ref{thm: bcvekhardyimpliescvekhardy}, if $\varphi_k \in H^p_{A,B}(D,\bc)$, for every $k$, then $\varphi_k = p^+ \varphi_k^+ + p^- \varphi_k^-$, for every $k$. Since $\varphi_k^+ \in H^p_{\gamma_k}(D)$, where $\gamma_k \in L^{\sigma_K}(D)$, $\sigma_k > 2$ or $1 < \sigma_k \leq 2$ and $p < \frac{\sigma_k}{2-\sigma_k}$, for every $k$, it follows, by Theorem \ref{bcvekhardybvfromidempotent}, that the nontangential boundary value $\varphi_{k,nt}$ of $\varphi_k$ exists, is in $L^p(\p D)$, and 
\[
    \lim_{r\nearrow 1}\int_0^{2\pi} ||\varphi_k(re^{i\theta}) - \varphi_{k,nt}(e^{i\theta})||^p_{\bc} \,d\theta = 0,
\]
for every $k$. The rest of the details are precisely the same as in the proof of Theorem \ref{thm: bchoivhardybvcon} after (\ref{bvconhoivk}).
\end{proof}

\begin{corr}
     For $n$ a positive integer, $p \geq 1$, and $A, B \in W^{n-1,q}(D,\bc)$, $q>2$, every $f \in V^{n,p}_{A,B}(D,\bc)$ with representation $f = \sum_{k=0}^{n-1} (\widehat{z^*} + \widehat{z})^k \varphi_k$ such that at least one of the following conditions holds:
     \begin{itemize}
     \item $\varphi_k = h_k + T_\bc(A\varphi_k + B\overline{\varphi_k})$, $\dbar \varphi_k \in L^{q_k}(D,\bc)$, $q_k > 2$, and $h_k^+ \in H^p_{\gamma_k}(D)$, $\gamma_k \in L^{\sigma_k}(D)$, $\sigma_k > 2$ or $1 < \sigma_k \leq 2$ and $p < \frac{\sigma_k}{2-\sigma_k}$, for every $k$, or
     \item $\varphi^+_k  \in H^p_{\lambda_k}(D)$, $\lambda_k \in L^{\nu_k}(D)$, $\nu_k > 2$ or $1 < \nu_k \leq 2$ and $p < \frac{\nu_k}{2-\nu_k}$, for every $k$,
     \end{itemize}
     has a boundary value in the sense of distributions $f_b$ and $f_b = f_{nt} \in L^p(\p D,\bc)$. 
\end{corr}

\section*{Acknowledgments}
The author thanks Briceyda B. Delgado for their hard work and support in producing the initial research on the HOIV-Hardy spaces \cite{WBD} and for our ongoing conversations about solutions to higher-order iterated Vekua equations. 

The author also thanks the anonymous referee for providing many helpful suggestions that have improved the quality of this article.

\printbibliography
\end{document}